\documentclass[a4paper,reqno,11pt]{amsart}
\usepackage{mathabx}
\usepackage{lmodern}
\usepackage[utf8]{inputenc}
\usepackage{amsfonts}
\usepackage{amsmath}
\usepackage{amssymb}
\usepackage{float}
\usepackage{tikz-cd}
\usepackage{graphicx}
\usepackage{bbm}
\usepackage{amscd}
\usepackage[sort&compress,numbers]{natbib}
\usepackage[left=1in,right=1in,bottom=1in,top=1in]{geometry}
\usepackage{etoolbox}
\usepackage{filecontents}
\usepackage{enumerate}
\usepackage[hidelinks]{hyperref}
\usepackage{tikz}

\usepackage{xspace}

\usetikzlibrary{matrix}
\newcommand{\ga}{\gamma}
\newcommand{\acts}{\curvearrowright}
\newcommand{\La}{\Lambda}
\newcommand{\sarpp}{$\mathrm{(SARP}^+\mathrm{)}$\xspace}

\newcommand{\eit}{\end{itemize}}

\newcommand{\mc}[1]{\mathcal{#1}}
\newcommand{\mbb}[1]{\mathbb{#1}}
\newcommand{\mr}[1]{\mathrm{#1}}

\newcommand{\mk}[1]{\mathfrak{#1}}
\newcommand{\wh}[1]{\widehat{#1}}
\newcommand{\mbf}[1]{\mathbf{#1}}
\newcommand{\mtt}[1]{\mathtt{#1}}
\newcommand{\wt}[1]{\widetilde{#1}}
\makeatletter
\patchcmd{\BR@backref}{\newblock}{\newblock(page~}{}{}
\patchcmd{\BR@backref}{\par}{)\par}{}{}
\makeatother
\theoremstyle{plain}
\newtheorem{theorem}{Theorem}[section]
\newtheorem*{theorem*}{Theorem}
\newtheorem{proposition}[theorem]{Proposition}

\newtheorem{corollary}[theorem]{Corollary}

\newtheorem{example}[theorem]{Example}

\newtheorem{lemma}[theorem]{Lemma}

\theoremstyle{definition}
\newtheorem{definition}[theorem]{Definition}

\theoremstyle{remark}

\newtheorem{remark}[theorem]{Remark}
\newtheorem{claim}{Claim}[theorem]
\newtheorem{pclaim}{Claim}[theorem]

\newcommand{\defi}{\begin{definition}}
\newcommand{\fdefi}{\end{definition}}
\newcommand{\eje}{\begin{example}}
\newcommand{\feje}{\end{example}}
\newcommand{\ejes}{\begin{ejemplos}}
\newcommand{\fejes}{\end{ejemplos}}
\newcommand{\lema}{\begin{lemma}}
\newcommand{\flema}{\end{lemma}}
\newcommand{\teor}{\begin{theorem}}
\newcommand{\fteor}{\end{theorem}}
\newcommand{\nota}{\begin{remark}}
\newcommand{\fnota}{ \end{remark}}
\newcommand{\clam}{\begin{claim}}
\newcommand{\fclam}{\end{claim}}
\newcommand{\pclam}{\begin{pclaim}}
\newcommand{\fpclam}{\end{pclaim}}
\newcommand{\clams}{\begin{claim*}}
\newcommand{\fclams}{\end{claim*}}
\newcommand{\prop}{\begin{proposition}}
\newcommand{\fprop}{\end{proposition}}
\newcommand{\cor}{\begin{corollary}}
\newcommand{\fcor}{\end{corollary}}

\newcommand{\lclam}{\begin{lclaim}}
\newcommand{\flclam}{\end{lclaim}}
\newcommand{\prucl}{\prue[Proof of Claim:]}
\newcommand{\fprucl}{\fprue}
\newcommand{\ben}{\begin{enumerate}}
\newcommand{\een}{\end{enumerate}}
\newcommand{\bit}{\begin{itemize}}

\newcommand{\N}{\mathbb N}
\newcommand{\R}{\mathbb R}
\DeclareMathOperator{\sma}{<\!}
\DeclareMathOperator{\red}{\mr{red}}
\DeclareMathOperator{\gl}{\mathrm{GL}}
\DeclareMathOperator{\age}{\mathrm{Age}}

\DeclareMathOperator{\ball}{\mathrm{Ball}}
\DeclareMathOperator{\Ball}{\mathrm{Ball}}
\DeclareMathOperator{\flim}{\mathrm{FLim}}
\DeclareMathOperator{\iso}{\mathrm{Iso}}
\DeclareMathOperator{\Emb}{\mathrm{Emb}}

\DeclareMathOperator{\Aut}{\mathrm{Aut}}
\DeclareMathOperator{\gr}{\mathrm{Gr}}
\DeclareMathOperator{\osc}{\mathrm{Osc}}

\newcommand{\rest}{\upharpoonright}

\newcommand{\C}{{\mathbb C}}

\newcommand{\id}{\mr{Id}}
\newcommand{\conj}[2]{ \{ {#1}\,:\,{#2} \} }
\newcommand{\prue}{\begin{proof}}
\newcommand{\fprue}{\end{proof}}
\newcommand{\Ga}{\Gamma}

\newcommand{\al}{\alpha}
\newcommand{\be}{\beta}
\newcommand{\de}{\delta}
\newcommand{\De}{\Delta}
\newcommand{\la}{\lambda}
\newcommand{\ro}{\varrho}
\DeclareMathOperator{\im}{\text{Im}}
\newcommand{\nrm}[1][\cdot]{\|#1\|}
\newcommand{\tnrm}[1][\cdot]{{\left\vert\kern-0.25ex\left\vert\kern-0.25ex\left\vert #1 
    \right\vert\kern-0.25ex\right\vert\kern-0.25ex\right\vert}}\newcommand{\con}{\subseteq}
\newcommand{\vep}{\varepsilon}
\newcommand{\buit}{\emptyset}

\newcommand{\om}{\omega}
\newcommand{\quo}{/\hspace{-0.1cm}/}
\begin{document}
\def\cprime{$'$}

\title[The  Ramsey property for grassmannians over $\R$]{The   Ramsey properties for grassmannians over $\R$, $\C$}
\author[D. Barto\v{s}ov\'{a}]{Dana Barto\v{s}ov\'{a}}
\address{Department of Mathematical Sciences, Carnegie Mellon University,
Pennsylvania, USA}
\email{dbartoso@andrew.cmu.edu}
\author[J. Lopez-Abad]{Jordi Lopez-Abad}
\address{Departamento de Matem\'{a}ticas Fundamentales,
Facultad de Ciencias, UNED, 28040 Madrid, Spain}
\email{abad@mat.uned.es}
\author[M. Lupini]{Martino Lupini}
\address{School of Mathematics and Statistics\\
Victoria University of Wellington\\
PO Box 600\\
Wellington 6140\\
New Zealand}
\address{Mathematics Department\\
California Institute of Technology\\
1200 E. California Blvd\\
MC 253-37\\
Pasadena, CA 91125}
\email{martino.lupini@vuw.ac.nz}
\urladdr{http://www.lupini.org/}
\author[B. Mbombo]{Brice Mbombo}
\address{Department of Mathematics and Statistics, University of Ottawa,
Ottawa, ON, K1N 6N5, Canada}
\email{bmbombod@uottawa.ca}
\thanks{D.B. was supported by the grant FAPESP 2013/14458-9. J.L.-A. \ was partially supported by the grant MTM2012-31286 (Spain) and the Fapesp Grant 2013/24827-1 (Brazil). M.L.\ was partially supported by
the NSF Grant DMS-1600186. B. Mbombo was supported by Funda\c{c}\~{a}o de
Amparo \`{a} Pesquisa do Estado de S\~{a}o Paulo (FAPESP) postdoctoral
grant, processo 12/20084-1. This work was initiated during a visit of
J.L.-A.\ to the Universidade de Sao P\~{a}ulo in 2014, and continued during
visits of D.B.\ and J.L.-A. to the Fields Institute in the Fall 2014, a visit
of M.L.\ to the Instituto de Ciencias Matem\'{a}ticas in the Spring 2015, and
a visit of all the authors at the Banff International Research Station in
occasion of the Workshop on Homogeneous Structures in the Fall 2015. The
hospitality of all these institutions is gratefully acknowledged.}

\begin{abstract}
 
In this note we study and obtain factorization theorems for
colorings of matrices and Grassmannians over $\mathbb{R}$ and ${\mathbb{C}}$%
, which can be considered metric versions of the Dual Ramsey Theorem
for Boolean matrices and of the Graham-Leeb-Rothschild Theorem for
Grassmannians over a finite field.

\end{abstract}

\maketitle

%\begin{center}
%{\bf \today }
%\end{center}

\section*{Introduction}

One of the most powerful principles in Ramsey theory is  the dual Ramsey theorem of R. L. Graham and B. L. Rothschild \cite{graham_ramseys_1971}. It trivially implies the classical Ramsey theorem or the much more involved Hales-Jewett Theorem. The  Dual Ramsey theorem is the particular instance of the   Rota's conjecture  for Grassmannians over the boolean field $\mbb F_2$, and it indeed implies the   Rota's conjecture for an arbitrary finite field, proved by Graham, Leeb and Rothschild (GLR) in \cite{graham_ramseys_1972}. These statements  can be categorized as a structural Ramsey theorem, the Dual Ramsey as the result   for finite Boolean algebras or for finite dimensional vector spaces over   the boolean field $\mbb F_2$, and  the (GLR) Theorem as  its natural generalization to finite dimensional vector spaces over an arbitrary finite field $\mbb F_p$. In this paper we  study    the  case of the infinite fields $\mbb F=\mbb R,\C$ in its metric form: Suppose that we endow the  $n$-dimensional vector space $\mbb F^n$ with a norm $\mtt m$.  We can naturally identify each $k$-dimensional subspace $V$ of $\mbb F^n$ with its unit ball $\ball(V,\mtt m)$, i.e.,   the centered section of $V$ with the $\ball(\mbb F^n,\mtt m)=\conj{v\in \mbb F^n}{\mtt m(v)\le 1}$. Thus, we can measure the distance between $V$ and $W$ by computing the Hausdorff distance $\La_{\mtt m}$ between the compact and convex sets $\ball(V,\mtt m)$ and $\ball(W,\mtt m)$.  Instead of trying to understand only discrete colorings $c:\gr(k,\mbb F^n)\to r:=\{0,1,\cdots,r-1\}$ we can  now work with 1-Lipschitz mappings, called here compact colorings,  $c:(\gr(k,\mbb F^n),\La_\mtt m)\to (K,d_K)$ into a compact metric space $(K,d_K)$ and ask how the restrictions of $c$ to Grassmannians $\gr(k,V)$  that are {\em congruent} to $\gr(k,\mbb F^m)$ look like. In this context,   a reasonable  notion of congruence  $\gr(k,V)\sim_\mtt m \gr(k,W)$  is  that  $(V,\mtt m)$ and $(W,\mtt m)$ are linearly isometric, or equivalently when there is an affine and symmetric bijection sending the     $(V,V')$-polar of $\ball(V,\mtt m)$ onto the $(W,W')$-polar of $\ball(W,\mtt m)$.  Notice that the set-mapping associated to a linear isometry from $V$ onto $W$ defines a $\La_\mtt m$-isometry from $\gr(k,V)$ onto $\gr(k,W)$. The corresponding quotient $\gr(k, \mbb F^n)/\sim_\mtt m$ is canonically identified with the class $\mc B_k(\mbb F^n,\mtt m)$ of isometric types of $k$-dimensional subspaces of $(\mbb F^n,\mtt m)$, a closed subset of the {\em Banach-Mazur compactum} $\mc B_k$.    In this paper we show that for the $p$-norms $\nrm[(a_j)_j]_p:=(\sum_{j}|a_j|^p)^{1/p}$, if $p\in [1,\infty[\setminus (2\N+4)$, and for the sup norm $\nrm[(a_j)_j]_\infty:=\max_j |a_j|$ we have that on each quotient $\mc B_k(\mbb F^n,\nrm_p)$ there is a compatible  ``Gromov-Hausdorff''-metric $\ga_p$, called here {\em extrinsic metric}, such that for every $k,m\in \N$ every compact metric space $(K,d_K)$ and every $\vep>0$ there is a dimension $n$ such that  for every compact coloring $c:(\gr(k,\mbb F^n),\La_{\nrm_p})\to (K,d_K)$ there is some $V\in \gr(m,\mbb F^n)$  that is ${\nrm_p}$-congruent to $\mbb F^m$ and there is a compact coloring $\wh c:(\mc B_k(\mbb F^n,\nrm_p), \ga_p)\to (K,d_K)$ such that $d_K(\wh{c}([W]_{\sim_\mtt m}), c(W))\le \vep$ for every   $W\in \gr(k,V)$.

In a similar way, we study factorizations of compact colorings of matrices of two kinds: $n\times k$-full rank matrices and $n$-square matrices of rank $k$, denoted by $M_{n,k}^k$ and by $M_n^k$, respectively.  When the field $\mbb F$ is finite, we show that for large enough $n$, for every coloring $c:M_{n,k}^k\to r$ there is some matrix $R\in M_{n,m}^m$ in {\em reduced column echelon form}  and a unique $\widehat{c}: \gl(\mbb F^k)\to r$ such that $c(R\cdot A)= \wh c(\red(A))$ for every $A\in M_{m,k}^k$, where $\mr{red}(A)$ is the $k$-square invertible matrix such that $A\cdot \red (A)$ is in reduced column echelon form.  We prove that  colorings of $M_{n}^k$ are factorized in a similar way by, in addition, using the full rank factorization of matrices.  We then analyze the colorings of these matrices over the fields $\mbb R,\mbb C$, and we compute the corresponding Ramsey factors in the metric context for the $p$-norms.

The proofs for the infinite fields are based on the crucial fact that when $\mtt m$ is a norm on the vector space $\mbb F^\infty$, the space of sequences $(a_n)_n$ with finitely many non-zero entries, have an approximate Ramsey property  called {\em steady approximate Ramsey property},  then there is a unique Banach space $\wh E$ such that $E:=(\mbb F^\infty,\mtt m)$ can be linearly isometrically embedded into $\wh E$,  $\mc B_k(E)$ is dense in $\mc B_k(\wh E)$, and such that the group $\iso(\wh E)$ of linear isometries of $\wh E$, with its strong operator topology, is {\em extremely amenable}, that is, every continuous action of $\iso(\wh E)$ on a compact space has a fixed point.  The corresponding spaces  to the $p$-norms are the Lebesgue space $L_p[0,1]$ if $p<\infty$, and the Gurarij space for the sup-norm. 
 
The use of tools from topological dynamics on a pure approximate Ramsey problem is not accidental. The recent {\em Kechris-Pestov-Todorcevic correspondence} in its discrete and metric versions characterizes the extreme amenability of the automorphism groups of Fraïssé (discrete/metric) structures in terms of the (approximate) Ramsey property of the collection of finitely generated substructures  (see \cite{kechris_fraisse_2005,melleray_extremely_2014,ferenczi_amalgamation_2019}).

The paper is organized as follows. We first study  Ramsey properties of matrices over  $\mbb F_2$ and then over an arbitrary finite field $\mbb F$. In particular, we provide in Theorem    \ref{finitefield}    another  proof of the Rota's conjecture as a straightforward consequence of the Dual Ramsey theorem.  To do this,  we use basic  tools from linear algebra, mainly the reduced column echelon form, that interestingly corresponds to some surjection being {\em rigid} with respect to the antilexicographical ordering, and that determines the Ramsey property (Proposition \ref{io43iuooi4343w44}).  We finish this section by introducing in Proposition \ref{8ret4433344}  a uniqueness principle for these Ramsey factorizations. 
The second section is devoted to the study of Ramsey factorizations   of matrices and Grassmannians over the fields $\R,\C$. We introduce the  main concepts, namely $\vep$-factors, and the Ramsey factors, including the extrinsic metrics, for full rank $n\times k$-matrices, Grassmannians, and $n\times n$-matrices of rank $k$,  and we present our main results in  Theorem \ref{factor_p_full_matrices}, Theorem \ref{factor_p_grass} and  Theorem \ref{factor_p_square_matrices}, respectively.    The third section is devoted to the proofs of the factorization results exposed  in section two. We recall the steady approximate Ramsey property \sarpp of a family of finite dimensional normed spaces  and the extreme amenability of  a topological group. We  explain in Corollary \ref{finitiz1}  when a normed space of the form $E=(\mbb F^\infty,\mtt m)$ has associated a unique   Banach space $\wh E$ that is Fraïssé, has a group of isometries that is extremely amenable, and how that gives Ramsey factors. In Subsection \ref{892sdfwe} we analyze these factors and we prove that they are the ones presented in Section two (Theorem \ref{Canonical_orbit_metrics}).   We finish with an appendix where we analyze the special case of the $\sup$-norm, and we give explicit definitions of   extrinsic metrics.    

\section{The Dual Ramsey Theorem and matrices over finite fields}\label{DRT_Matrices}

To keep the notation unified, let $\mbb F^\infty$ be the vector space over $\mbb F$ consisting of all eventually zero sequences $(a_n)_{n\in \N}$. Let $(u_n)_{n\in \N}$ be the {\em unit basis} of $\mbb F^\infty$, that is, each $u_n$ is the sequence whose only non-zero entry is 1 at the $n^\mr{th}$-coordinate. In this way we   identify $\mbb F^n$ with the subspace $\langle u_j\rangle_{j<n}$ of $\mbb F^\infty$, and then $\mbb F^\infty$ with the increasing union of all $\mbb F^n$.

Given $\al,\be\in \N\cup \{\infty\}$, let $M_{\al, \be}(\mbb F)$ be the collection of $\al \times \be$-matrices  with finitely many non-zero entries. In a similar manner as before, $M_{\al,\be}(\mbb F)= \bigcup_{ n\le \al, m\le \be} M_{n,m}(\mbb F)$, increasing union.  Let ${M}_{\al, \be}^k(\mathbb{F})$ be the set of all  $\al\times \be$-matrices of rank $k$ with entries
in $\mathbb{F}$
 To lighten the notation, when there  is no possible confusion, we will write $M_{\al,\be}$, $M_{\al,\be}^k$,... to denote $M_{\al,\be}(\mbb F)$, $M_{\al,\be}^k(\mbb F)$,...

There are several equivalent ways to present the dual Ramsey theorem (DRT) of Graham and Rothschild \cite{graham_ramseys_1971}. Among these, there is a factorization result for \emph{Boolean matrices} stated below as Theorem \ref{Booleanmatrices}. Motivated by this, we study  Ramsey-theoretical factorization results for colorings of other classes of matrices. We begin with matrices with entries in a finite field, and then conclude, in the next section, with matrices over $\mathbb{R}$ or ${\mathbb{C}}$.

It is well known, for example using the Gauss-Jordan elimination method, that an  $n\times m$-matrix $A$  has a unique decomposition $A=\mr{red}(A)\cdot \tau(A)$ where $\mr{red}(A)$    is in reduced column echelon   form and $\tau(A)$ is an invertible $m\times m$-matrix. We prove that when the field is finite any finite coloring of matrices over a finite field is determined, in a precise way, by $\tau$. This  can be seen as an extension  of the well known result of Graham, Leeb, and Rothschild on Grassmannians over a finite field \cite{graham_ramseys_1972}.   
\defi[Factors]
Let $X$ be a set and $r\in \N$. An \emph{$r$-coloring} of $X$ is a mapping $c: X\to r=\{0,1\dots,r-1\}$. A subset $Y$ of $X$ is \emph{$c$-monochromatic} if $c$ is constant on $Y$.  We say that  a mapping $\pi:X\to K$ is a \emph{factor} of  $c: X\to r$  if there is some $\widetilde{c}: K\to r$ such that $c= \widetilde c \circ \pi$.  Finally, $\pi$ is a \emph{factor of $c$ in $Y\con X$} if $\pi\upharpoonright _{Y}$ is a factor of $c\upharpoonright _{Y}$. So, $Y$ is $c$-monochromatic  when  the trivial constant map $\pi:X\to \{0\}=1$ is a factor of $c$  in $Y$.   
\fdefi

We now  recall the \emph{Dual Ramsey Theorem }(DRT) of Graham and
Rothschild \cite{graham_ramseys_1971} (see also  \cite{nesetril_handbook}, \cite{todorcevic_2010}). For convenience, we present its
formulation in terms of rigid surjections between finite linear orderings.
Given two linear orderings $(R,<_{R})$ and $(S,<_{S})$, a surjective map $%
f:R\rightarrow S$ is called a \emph{rigid surjection} when $\min
f^{-1}(s_{0})<_R\min f^{-1}(s_{1})$ for every $s_{0},s_{1}\in S$ such that $%
s_{0}<_{S}s_{1}$. We let $\mathrm{Epi}(R,S)$ be the collection of rigid
surjections from $R$ to $S$.

\begin{theorem}[Graham--Rothschild]
\label{Theorem:DLT-rigid}For every finite linear orderings $R$ and $S$ such
that $\#R<\#S$ and every $r\in \mathbb{N}$ there exists an integer $n>\#S$
such that, considering $n $ naturally ordered, every $r$-coloring of $%
\mathrm{Epi}(n,R)$ has a monochromatic set of the form $\mathrm{Epi}%
(S,R)\circ \gamma =\{{\sigma \circ \gamma }\,:\,{\sigma \in \mathrm{Epi}(S,R)%
}\}$ for some $\gamma \in \mathrm{Epi}(n,S)$.
\end{theorem}

\subsection{Ramsey properties of colorings of Boolean matrices}
	
	Perhaps the most common formulation of the {\em dual Ramsey Theorem} of Graham and
Rothschild  is done in terms of partitions. Given $k,m,n\in \mathbb{N}$, let $\mathcal{E}_m(n)$ be the set of all partitions of $n$ into $m$ pieces. Given   $\mathcal{P}\in  \mathcal{E}_m(n)$, let $\langle \mathcal{P}\rangle_k$ be  the set of all partitions $\mathcal{Q}$ of $n$ with $k$ pieces that are coarser than $\mathcal{P}$, i.e., such that each piece of $\mathcal{Q}$ is a union of pieces of $\mathcal{P}$.

\begin{theorem*}[DRT, partitions version]
For every $k,m\in \mathbb{N}$ and $r\in \mathbb{N}$ there is  $n\in \N$ such that every $r$-coloring of $\mathcal{E}_k(n)$ has a
monochromatic set of the form $\langle \mathcal{P}\rangle_k$ for some $\mathcal{P}\in \mathcal{E}_m(n)$.
\end{theorem*}
 The following three reformulations of the Dual Ramsey Theorem are \emph{structural Ramsey results} for finite Boolean
algebras.

\begin{theorem*}[DRT, Boolean algebras]
Let $\mathcal{A}$ and $\mathcal{B}$ be finite Boolean algebras, and let $r\in \mathbb{N}$. Then there exists a finite Boolean algebra $\mathcal{C}$
such that every $r$-coloring of the set $\binom{\mathcal{C}}{\mathcal{A}}$
of isomorphic copies of $\mathcal{A}$ inside $\mathcal{C}$ admits a
monochromatic set of the form $\binom{\mathcal{B}_{0}}{\mathcal{A}}$ for
some $\mathcal{B}_{0}\in \binom{\mathcal{C}}{\mathcal{B}}$.
\end{theorem*}

Let $\mathcal{A}$ be a finite Boolean algebra.  Any  $a\in \mc A$  is represented as
\begin{equation*}
a=\bigvee_{x\in \Gamma _{a}}x,
\end{equation*}
for a unique set of atoms $\Gamma _{a}$. So, any linear ordering $<$ on the sets of atoms $\mathrm{At}(\mathcal{A})$ of $\mc A$ \emph{extends} to $\mathcal{A}$ by defining $a<b$ iff $\min_{<}(\Gamma _{a}\triangle \Gamma _{b})\in \Gamma _{a}$. Following  \cite{kechris_fraisse_2005}, we will say that $(\mathcal{A},<)$ is a \emph{canonically ordered (c.o.)} Boolean algebra. Given c.o. Boolean algebras $\mathcal{A}$ and $\mathcal{B}$, let $\mathrm{Emb}_{<}(\mathcal{A},\mathcal{B})$ be the collection of ordering-preserving embeddings from $\mathcal{A}$ into $\mathcal{B}$, respectively.

\begin{theorem}[DRT, canonically ordered Boolean algebras]
\label{Theorem:DRT-ordered-Boolean}Given c.o.   Boolean algebras $\mathcal{A}$ and $\mathcal{B}$ and    $r\in \N$, there is a c.o. Boolean algebra $\mathcal{C}$ such that  each $r$-coloring  of $\mathrm{Emb}_{<}(\mathcal{A},\mathcal{C})$ has a monochromatic set of the form  $\varrho \circ \mathrm{Emb}_{<}(\mathcal{A},\mathcal{B})$    for some $\varrho \in \mathrm{Emb}_{<}(\mathcal{B},\mathcal{C})$.
\end{theorem}

Suppose that $\mathcal{A}$ and $\mathcal{B}$ are finite Boolean algebras
with $k$ and $n$ atoms, respectively. Any   embedding from $\mathcal{A}$ to $\mathcal{B}$ has a corresponding \emph{representing }$n\times k$ matrix with entries in $\left\{ 0,1\right\} $. We call the matrices arising in this fashion \emph{Boolean matrices}. The set of $n\times k$
Boolean matrices will be denoted by ${M}_{n, k}^{\mr{ba}}$, i.e., the set of $n\times k$ matrix with entries in $\left\{
0,1\right\} $ whose columns (which can be identified with subsets of $n$)
form a $k$-partition of $n$. We let ${M}_{n,k}^{\mr{oba}}$ be the set of Boolean $n\times k$-matrices
that correspond to order-preserving embeddings between c.o. Boolean algebras. These are precisely the set of Boolean matrices
whose columns $( P_{i}) _{i\in k}$ furthermore satisfy $\min
P_{i}<\min P_{i+1}$ for $i< k-1$.

In the following, we identify a permutation $\sigma $ of $k$ with the
associated $k\times k$ permutation matrix. This allows one to identify the
group $\mathcal{S}_{k}$ of permutations of $k$ with a group of unitary
matrices. Let $\pi :{M}_{n,k}^{\text{ba}}\rightarrow \mathcal{S}_{k}$ be the function assigning to a matrix $A$ the unique element $\pi (A)
$ of $\mathcal{S}_{k}$ such that $A=A_< \cdot \pi (A)$ for some (uniquely determined) matrix $A_<\in M_{n,k}^{\mathrm{oba}}$. Given an $n\times m$-matrix $A$, we let  $A\cdot {M}_{m, k}^{\text{ba}}=\{{A\cdot B}\,:\, {B\in {M}_{m, k}^{\text{ba}}}\}$.

\begin{theorem}[DRT, Boolean matrices]
\label{Booleanmatrices} For every $k,m \in \N$ and $r\in \N$ there is $n$ such that
for every $c:{M}_{n, k}^\mathrm{ba}\to r$ there is $R\in
{M}_{n, m}^\mathrm{oba}$ such that  $\pi$ is a factor of $c$ in   $R \cdot  M_{m,k}^\mathrm{ba}$.
That is, the color of $R\cdot B$ depends only on $\pi( B) = \pi (R\cdot B)$ for every $B\in
{M}_{m, k}^{\mathrm{ba}}$.
\end{theorem}

\begin{proof}
Let $\mathcal{C}$ be a c.o.  Boolean algebra obtained
by applying the Dual Ramsey Theorem for c.o. Boolean algebras---Theorem \ref{Theorem:DRT-ordered-Boolean}---to the power sets $\mathcal{P}(k)$, $\mathcal{P}(m)$  canonically ordered as above by $s<t$ if and only if $\min (s\triangle t)\in s$, and to the number of colors $r^{\mathcal{S}_{k}}$. Without loss of
generality we can assume that $\mathcal{C}$ is equal to $\mathcal{P}(n) $ for some $n\in \omega $, since any c.o.
Boolean algebra is of this form. We claim that such an $n$ satisfies the
desired conclusions. Indeed, fix a coloring $c:{M}_{n, k}^{\text{ba}}\rightarrow r$. This induces a coloring $f:\mathrm{Emb}_{<}(\mathcal{P}(k),\mathcal{P}(n))\rightarrow r^{\mathcal{S}_{k}}$ as follows.
Let $\gamma $ be an element of $\mathrm{Emb}_{<}(\mathcal{P}(k),\mathcal{P}(n))$, and let $A_{\gamma }\in {M}_{n,k}^{\text{ba}}$ be the
corresponding representing matrix. Define then $f( \gamma ) $ to
be the element $( c( A_{\gamma }\cdot \sigma ) ) _{\sigma
\in \mathcal{S}_{k}}$ of $r^{\mathcal{S}_{k}}$. By the choice of $\mathcal{C}=\mathcal{P}( n) $ there exists $\varrho \in \mathrm{Emb}_{<}(\mathcal{P}(m),\mathcal{P}(n))$  such that $f$ is constant on $\varrho
\circ \mathrm{Emb}_{<}(\mathcal{P}(k),\mathcal{P}(m))$. Let now $\widetilde{c}\in r^{\mathcal{S}_{k}}$ be the constant value of $f$. It is now easy to see that $c(A_\ro \cdot B)=\widetilde c (\pi(B))$ for every $B\in M^\text{ba}_{m, k}$.
\end{proof}

\subsection{Ramsey properties of colorings of matrices over a finite field}
It is natural to consider Ramsey properties of other classes of matrices over a
field $\mathbb{F}$. We are going to see that for $\mathbb{F}$ finite there
is a factorization result similar to the DRT for Boolean matrices, that
extends the well known theorem by Graham, Leeb and Rothschild on
Grassmannians  $\mathrm{Gr}(k,V)$,  the family of all $k$-dimensional subspaces of a vector space $V$
over $\mathbb{F}$. In the following, given a sequence $(x_i)$ in a  vector space $E$, we let $\langle x_i \rangle $ be its linear span inside $E$. 

\begin{theorem}[Graham-Leeb-Rothschild \protect\cite{graham_ramseys_1972}]
\label{Theorem:GLR} Given  $k,m, r\in \N$ there is  $n\in \N$
such that every $r$-coloring of  $\mathrm{Gr}(k,\mathbb{F}^{n})$ has a monochromatic set of the form $\mathrm{Gr}(k,R)$ for some $R\in
\mathrm{Gr}(m,\mathbb{F}^{n})$.
\end{theorem}

This result is a particular case of the factorization theorem for injective
matrices. Recall that a $p\times q$-matrix $A=(a_{ij})$ is in \emph{reduced row echelon form} (RREF) when there is $p_{0}\leq p$ and (a unique) strictly increasing sequence $(j_{i})_{i<p_{0}}$ of integers $<q$ such that
\begin{enumerate}[i)]
\item $A\cdot u_{j_{i}}=u_{i}$ for every $i<p_{0}$ and

\item $\langle A\cdot u_{j}\rangle _{j<j_{i}}=\langle u_{l}\rangle _{l<i}$ for every $i<p_{0}$.
\end{enumerate}
When $A$ is
in RREF  and it has rank $p$, we define  $I_A$ as the  $q\times p$-matrix  with entries in $\left\{
0,1\right\} $, and whose nonzero entries are in the positions $(j_{i},i)$ ($i<p$).  For example for the field $\mbb F_5$ and 
\begin{equation}\label{lk3jreijwrijw}
 A= \left( \begin{array}{cccccc}  1 & 2 & 0 & 3 &0 &1 \\ 0 & 0 & 1&4 & 0 &2 \\ 0 & 0 & 0 & 0&1 & 3
\end{array} \right) \text{ we have  }I_A= \left( \begin{array}{ccc}   1 & 0 & 0 \\ 0 & 0 & 0 \\ 0& 1 & 0\\0 & 0 & 0 \\ 0& 0& 1 \\0 & 0 & 0 
\end{array} \right)  
\end{equation}    

It follows   that $I_A$ is a right inverse to $A$, i.e., $A\cdot I_{A}=\mathrm{Id}_{p}$. A matrix $A$ is in \emph{reduced column echelon form (RCEF)} when
its transpose $A^{\mathrm{t}}$ is in RREF. Let $\mathcal{E}_{n,m}(\mathbb{F})$, $\mc E(\mbb F)$   be the collection
of $n\times m$-matrices of rank $m$  in RCEF  and  of  full rank matrices in RCEF, respectively.  

\defi
Let $\tau:  M_{\infty, k}^k\to \mr{GL}(\mbb F^k)$ be the mapping that assigns to each
$A\in {M}_{\infty, k}^k(\mathbb{F})$  the unique $k\times k$-invertible matrix $\tau (A)$  such that $A\cdot \tau (A)$ is in  RCEF. Let also $\mr{red}_c(A):=A\cdot \tau (A)$.
\fdefi

\begin{theorem}[Factorization of colorings  of full rank matrices over a finite field]
\label{finitefield} Given $k,m,r\in \N $ there is $n\in \N $ such
that for every  $c:{M}_{n, k}^k(\mathbb{F})\to r$  there is $R\in \mathcal{E}_{n, m}(\mathbb{F})$ such that $\tau$ is a factor  of $c$ in  $ R \cdot M_{m, k}^k(\mbb F)  $.
 \end{theorem}
This gives immediately the Graham-Leeb-Rothschild
Theorem---Theorem \ref{Theorem:GLR}---as every $k$-dimensional subspace of $\mathbb{F}^{n}$ can be represented as the linear span of the columns of a
matrix in RCEF. The proof of Theorem \ref{finitefield} is a direct consequence of the DRT and the next propositions. In the following, we fix an ordering $<$ on the finite field $\mathbb{F}$ such that $0<1$ are the first two elements of $\mathbb{F}$. We let $\mathbb{F}^{k}$ be endowed with
the corresponding antilexicographic order $<_{\mathrm{alex}}$ and we define  $\Phi
_{n,k}:\mathrm{Epi}(n,\mathbb{F}^{k})\rightarrow {M}_{n,
k}^k$ as the function assigning to each rigid surjection $f$ the matrix whose
rows are $f(j)$ for every $j< n$.

\begin{lemma}
\label{io43iuooi4343w44dsw} $\Phi _{n,k}(f)$ is a full rank matrix in
RCEF.
\end{lemma}

\begin{proof}
It is clear that $\Phi _{n, k}( f) $ is a full rank matrix.
We prove that it is in RCEF. Let $A$ be the transpose of $\Phi _{n, k}(f)$. For each $i\in k$, let $j_{i}:=\min \{{j<n}\,:\, {A\cdot u_{j}=u_{i}}\}$. Then $(j_{i})_{i<k}$ is strictly
increasing, since $f$ is a rigid surjection, and if $j<j_{i}$, then $A\cdot
u_{j}<_{\mathrm{alex}}u_{i}$, by the definition of $j_{i}$, and the rigidity of $f$. Therefore $A\cdot u_{j}\in \langle u_{l}\rangle
_{l<i}$.  Consequently, $A$ is in RREF. \end{proof}
The next is the key relation between matrices in RREF and rigid surjections that will allow us to use the dual Ramsey Theorem and prove Theorem  \ref{finitefield}. 
\begin{proposition}
\label{io43iuooi4343w44} For $A\in {M}_{k, n}^k(\mathbb{F})$ the following are equivalent.

\begin{enumerate}[i)]
\item $A$ is in RREF.

\item The linear map $T_A:\mathbb{F}^{n}\rightarrow \mathbb{F}^{k}$
represented by $A$ in the corresponding unit bases is a rigid surjection and for every $i<k$ there is a column of $A$ equal to $u_{i}$.
\end{enumerate}
\end{proposition}
In particular we have the following.
\begin{corollary}
Suppose that $A\in M_{n,m}^m(\mbb F)$ and $B\in M_{m,k}^k(\mbb F)$.
\begin{enumerate}[a)]
\item  If $A$ and $B$ are in RCEF (resp. RREF) then $A\cdot B$ is also in RCEF (resp. RREF).
\item If $A$ is in RCEF then $\tau(A\cdot B)= \tau(B)$. \qed
  
\end{enumerate}  
\end{corollary}

\begin{proof}[Proof of Proposition \ref{io43iuooi4343w44}]
{\it i)}$\Rightarrow ${\it ii)} Suppose that $A$ is in RREF.  We will prove that the canonical linear operator $T_A:\mathbb{F}^{n}\rightarrow \mathbb{F}^{k}$, $T_A(u_{i}):=A\cdot u_{i}$  for  $i<n$ is a rigid surjection from $\mathbb{F}^{n}$ to $\mathbb{F}^{k}$ endowed with the antilexicographical order $<_{\mathrm{alex}}$ described before. Let $(j_{i})_{i<k}$ be the strictly increasing sequence in $n$ witnessing that $A$ is in RREF. By linearity, $T_A(0)=0$. Fix now $w\in \mathbb{F}^{k}$.

\begin{pclaim}
$\min_{<_{\mathrm{alex}}}(T_A)^{-1}(w)=I_{A}\cdot w$.
\end{pclaim}
From this, since $I_{A}:\mathbb{F}^{k}\rightarrow \mathbb{F}^{n}$ is  $<_{\mathrm{alex}}$-increasing, we obtain that $T_A$ is a rigid surjection.

\begin{proof}[Proof of Claim:]  Applied to the example in \eqref{lk3jreijwrijw} and to $w=(1,2,3)$,  it should be clear that the spread $I_A \cdot (1,2,3)= (1,0,2,0,3,0)$ of $(1,2,3)$ is the $<_\mr{alex}$-least  element of the preimage of $(1,2,3)$ under $T_A$. We give a detailed proof. 
Suppose that $(v_{j})_{j<n}=\bar{v}=\min_{<_{\mathrm{alex}}}\{{v\in
\mathbb{F}^{n}}\,:\,{A\cdot v=w}\}$. Set $z=(z_{j})_{j}:=I_{A}(w)$. We prove
by induction on $i<k$ that $v_{j}=z_{j}$ for every $j\geq j_{k-i-1}$.
Suppose that $i=0$. Since for every $j>j_{k-1}$ one has that $z_{j}=0$, we
obtain that $v_{j}=0$, by $<_{\mathrm{alex}}$-minimality of $\bar{v}$. Let $(A)_{k-1}$ be the $(k-1)^{\mathrm{th}}$-row of $A$. It follows that $(A)_{k-1}=u_{j_{k-1}}+y$, where $y\in \langle u_{j}\rangle _{j>j_{k-1}}$. Hence,
\begin{equation*}
z_{j_{k-1}}=w_{k-1}=(A)_{k-1}\cdot \bar{v}=v_{j_{k-1}}.
\end{equation*}
Suppose that the conclusion holds for $i$, that is, $v_{j}=z_{j}$ for every $j\geq j_{k-i-1}
$. We will prove that it also holds for $i+1$. Since $v\leq _{\mathrm{alex}}z$, and $z_{j}=0$ for every $j_{k-i'-2}<j<j_{k-i'-1}$ and $0\leq i'\leq i$, we obtain that $v_{j}=0$ for
such $j$'s. Then the $(k-i-2)^{\mathrm{nd}}$ row of $A$ is of the form $(A)_{k-i-2}=u_{j_{k-i-2}}+y$ with $y$ in the   span of $\{{u_{j}}\,:\,{j>j_{k-i-2},\,j\neq j_{p}\text{ for all $p$}}\}$. 
 It follows that
\[\pushQED{\qed}
z_{j_{k-i-2}}=w_{k-i-2}=(A)_{k-i-2}\cdot \bar{v}=v_{j_{k-i-2}}. \qedhere \popQED
\] \let\qed\relax
\end{proof}
{\it ii)}$\Rightarrow ${\it i)} Now suppose that $T_A$ is a rigid surjection from $\mathbb{F}^{n}$ to $\mathbb{F}^{k}$ with respect to the antilexicographical
orderings, and that for every $i<k$ a column of $A$ is $u_{i}$. For each $i<k$, let $j_{i}$ be the first such column of $A$. We prove that $(j_{i})_{i<k}$ witnesses that $A$ is in RREF, that is:

\begin{pclaim}
$T_A\langle u_j\rangle_{j<j_i}= \langle u_l\rangle_{l<i}$ for every $i<k$.
\end{pclaim}

\begin{proof}[Proof of Claim:]
The proof is by induction on $i$. If $i=0$, then $T_A\langle
u_j\rangle_{j<j_0}=\{ 0\} $ because $u_0$ is the second element of $\mathbb{F}^n$ in the antilex ordering, while the first element is the zero vector. Suppose the result is true for $i$, and let us
extend it to $i+1$. In particular, we know that $j_{i+1}>j_i$, and it is clear that $\langle u_l\rangle_{l\le i}\subseteq T_A\langle u_j\rangle_{j\le
j_i}\subseteq \langle u_j\rangle_{j<j_{i+1}}$. Suppose towards a contradiction that there exists $j$ such that $j_i<j<j_{i+1}$ and $T_A(u_j)\notin \langle u_l\rangle_{l\le i}$. Denote by $\xi$ the least such $j$. Thus, $u_{i+1}\le_{\mathrm{alex}} T_A(u_\xi)$, hence there is some $x\le_\mathrm{alex} u_\xi$ such that $T_A(x)=u_{i+1}$. This means, by the minimality of $\xi$, that $T_A(u_\xi)= y+u_{i+1}$ with $y\in \langle u_l\rangle_{l\le i}$. We know that $y\neq 0$ by the minimality of $j_{i+1}$; so $u_{i+1}<_\mathrm{alex} y+u_{i+1}$. Hence, \begin{equation*}
\min (T_A)^{-1}(u_{i+1})<_\mathrm{alex} \min (T_A)^{-1}(y+u_{i+1})=u_\xi.
\end{equation*}
So, there is $x\in \langle u_j\rangle_{j<\xi}$ with $T_A(x)=u_{i+1}$, which is impossible by the minimality of $\xi$. \qedhere   
\popQED  \popQED 
 \let\qed\relax\end{proof}
\end{proof}

\begin{proof}[Proof of Theorem \protect\ref{finitefield}]
Fix all parameters. We consider $\mathbb{F}^{k}$ and $\mathbb{F}^{m}$ antilexicographically ordered by $<_{\mathrm{alex}}$ (as explained before). Let $n$ be obtained from the linear orderings $(\mathbb{F}^{k},<_{\mathrm{alex}})$, $(\mathbb{F}^{m},<_{\mathrm{alex}})$ and the number of colors $r^{\lambda }$, where $\lambda =\prod_{i=0}^{k-1}(p^{k}-p^{i})$ is the order of the group $\mathrm{GL}(\mathbb{F}^{k})$, by applying the Dual Ramsey Theorem for rigid surjections (Theorem \ref{Theorem:DLT-rigid}). We claim that $n$ satisfies the desired conclusions. Fix a coloring $c:{M}_{n, k}^k(\mathbb{F})\rightarrow r$. Let $c_{0}:\mathrm{Epi}(n,\mathbb{F}^{k})\rightarrow r^{\mathrm{GL}(\mathbb{F}^{k})}$ be the coloring defined by $c_{0}(\sigma):=(c(\Phi _{k,n}(\sigma )\cdot \Gamma ^{-1}))_{\Gamma \in \mathrm{GL}(\mathbb{F}^{k})}$ for $\sigma \in \mathrm{Epi}(n,\mathbb{F}^{k})$. By the choice of $n$, there exists $\varrho \in \mathrm{Epi}(n,\mathbb{F}^{m})$
such that $c_{0}$ is constant on $\mathrm{Epi}(\mathbb{F}^{m},\mathbb{F}^{k})\circ \varrho $ with constant value $\widetilde{c}\in r^{\mathrm{GL}(\mathbb{F}^{k})}$. Let $R:=\Phi _{n,m}(\varrho )$. We claim that $R$ and $\widetilde{c}$ satisfy the conclusion of the statement in the theorem. It follows from Proposition \ref{io43iuooi4343w44dsw} that $R\in \mathcal{E}_{n, m}(\mathbb{F})$. Now let $A\in {M}_{m, k}^k(\mathbb{F})$. We have to prove that $c(R\cdot A)=\widetilde{c}(\tau (R\cdot A))$. First, note that $\tau (R\cdot A)=\tau (A)$, because $R$ is in RCEF. Let $B$ be the transpose
of $\mathrm{red}_{c}(A)$ (i.e., $B$ is the RREF of the transpose of $A$), and let $T_B:\mathbb{F}^{m}\rightarrow \mathbb{F}^{k}$ be the linear operator
defined by $B$ in the corresponding canonical bases. We know by Proposition \ref{io43iuooi4343w44} that $T_B\in \mathrm{Epi}(\mathbb{F}^{m},\mathbb{F}^{k})$.

\pclam
$\Phi_{n,k}(T_B \circ \ro)= R \cdot \mr{red}_c(A)$.
\fpclam
\prucl
Fix $j<m$. Then the $j^\mr{th}$-row $(\Phi_{n,k}(T_B \circ \ro))_j$ of $\Phi_{n,k}(T_B \circ \ro)$ is the row vector $T_B(\ro(j))$. Hence,
\[ \pushQED{\qed}
(\Phi_{n,k}(T_B \circ \ro))_j= T_B(\ro(j))= ((\mr{red}_c(A))^\mr{t} \cdot ((R)_j)^\mr{t})^\mr{t}=(R)_j \cdot \mr{red}_c(A)=(R \cdot \mr{red}_c(A))_j.  \qedhere \popQED
\]  \let\qed\relax
\fprucl
So, given $\Ga\in \mr{GL}_k(\mbb F)$ we have that
\[ \pushQED{\qed}
c(R\cdot A)=c( R \cdot \mr{red}_c A \cdot \tau(A)^{-1} ) =(c_0( R \cdot \mr{red}_c A))(\tau(A))=\widetilde{c} (\tau(A))=\widetilde{c} (\tau(R\cdot A)).  \qedhere \popQED
\]  \let\qed\relax
\fprue

\subsubsection{Square matrices of rank $k$}
We present the Ramsey factorization for finite colorings of square matrices. Recall that every $n\times m$-matrix $A$ of rank $k$ has a {\em full rank decomposition} $A=B \cdot C$ where $B\in M_{n,k}^k$ and $C\in M_{k,m}^k$.  
\defi
Given $k$ and $n$, let $\tau^{(2)}: M_{n, n}^k\to \mr{GL}(\mbb F^k)$ be the mapping uniquely defined by the relation $A= A_0 \cdot \tau^{(2)}(A) \cdot A_1^\mr{t}$ for some $A_0,A_1\in \mathcal E_{n, k}(\mathbb{F})$.
\fdefi
It is routine to see that $\tau^{(2)}$ is well defined.
\teor[Factorization of colorings  of square matrices over a finite field]
For every   $k,m,r\in \N$  there is $n\in \N$ such that for  every $c:M_{n, n}^k(\mbb F)\to r$  there are $R_0,R_1\in \mathcal E_{n, m}(\mathbb{F})$ such that $\tau^{(2)}$ is a factor of $c$  in $R_0\cdot   M_{m , m}^k(\mbb F) \cdot R_1^\mr{t} $.
%, i.e., the $c$-coloring of $R_0\cdot A \cdot R_1^\mr{t}$ only depends on $\tau^{(2)}(A)$.
 \fteor
\prue  
Given integers $k,m$ and $r$, let $n_\mbb F(k,m,r)$ be the minimal number $n$ such that the factorization statement in Theorem \ref{finitefield} holds for the parameters $k,m$ and $r$, and now let  $n_0:=n_\mbb F(k,m,r^{\mr{GL}(\mbb F^k)})$, and let $n:=n_\mbb F(k,n_0, r^{M^k_{n_0,k}}(\mathbb{F}))$.  We claim
that $n$ works. Fix any $r$-coloring $f: M_{n, n}^k(\mathbb{F})\to r$ and $P\in \mathcal E_{n, n_0}(\mathbb{F})$. We define the
coloring $c:M^k_{n, k}(\mathbb{F})\to r^{ M^k_{n_0, k}(\mathbb{F})}$ by
$$c(A):= (f(A \cdot B^\mr{t} \cdot P^\mr t))_{B\in M^k_{n_0,k}(\mathbb{F})}.$$
The coloring $c$ is well defined because $A\cdot B^\mr{t} \cdot P^\mr{t}$ has  rank $k$. Let $R\in \mathcal
E_{n, n_0}$  and $c_0:\mr{GL}(\mbb F^k)\to r^{M^k_{n_0, k}(\mathbb{F})}$ be such that $c(R \cdot
A)=c_0(\tau(A))$ for  $A\in  M^k_{n_0, k}(\mathbb{F})$. Define  $d:M^k_{n_0, k}(\mathbb{F})\to r^{\mr{GL}(\mbb
F^k)}$ by
$$d(B):=(c_0(\Ga)(B))_{\Ga\in \mr{GL}(\mbb F^k)}.$$
Let $S\in \mathcal E_{n_1, m}(\mathbb{F})$ and $d_0:\mr{GL}(\mbb F^k)\to r^{\mr{GL}(\mbb F^k)}$  be such that $d(S\cdot
B)= d_0(\tau(B))$ for every $B\in  M^k_{m, k}(\mathbb{F})$.
 Set   $R_0= R\cdot Q$ and   $R_1:= P \cdot S$, where
$Q\in \mathcal E_{n_0, m}(\mathbb{F})$ is arbitrary.  Finally, let $g:\mr{GL}(\mbb F^k)\to r$ be defined by $g(\Ga)=
d_0(\Ga_0)(\Ga_1)$, where $\Ga=\Ga_0 \cdot \Ga_1^\mr{t}$ are arbitrary. Notice that if $\Ga=\Ga_1 \cdot
\Ga_0^\mr{t}$, then it follows that
\begin{align*}
d_0(\Ga_0)(\Ga_1)= &d(S \cdot P_0 \cdot \Ga_0)(\Ga_1)=c_0(\Ga_1)(S \cdot P_0 \cdot \Ga_0)=c(R\cdot P_1 \cdot \Ga_1)(S \cdot P_0 \cdot \Ga_0) =\\
=&f(R\cdot P_1 \cdot \Ga_1 \cdot  \Ga_0^{\mr{t}}  \cdot P_0^{\mr{t}} \cdot R_1^{\mr{t}} )=f(R\cdot P_1 \cdot \Ga  \cdot P_0^{\mr{t}} \cdot R_1^{\mr{t}} )
\end{align*}
where $P_0\in \mathcal E_{m, k}(\mathbb{F})$ and $P_1\in \mathcal E_{n_0, k}(\mathbb{F})$ are arbitrary. So,  $g$ does not
depend on the  decomposition $\Ga=\Ga_1\cdot\Ga_0^\mr{t}$. Similarly one proves that $g(\tau^{(2)}(A))= f(R_0 \cdot A \cdot  R_1^\mr{t})$
for all $A\in M_{m, m}^k(\mathbb{F})$.
\fprue
\subsubsection{Uniqueness} 
We see that in a natural way the factors we presented are unique. We introduce the abstract notion of Ramsey factor in this context. 
\begin{definition}\label{lijeriowijorewidse}
Given $\mu: M_{\infty,k}^k(\mc F)\to X$, $X$ finite, and $\mc A\con \bigcup_{n,m}M_{n,m}(\mc F)$,  we say that the couple $(\mu,\mc A)$ is a {\em $k$-Ramsey factor}  when
\begin{enumerate}[i)]
\item $\mu(M_{\infty,k}^k(\mc F))=X$.
\item   $\mu(R\cdot A)=\mu(A)$ for every $A\in M_{m,k}^k(\mc F)$ and every $R\in \mc A\cap M_{n,m}^m(\mc F)$.            
\item   For every $m,r\in \N$ there is some $n\in \N$ such that for every $r$-coloring  $c$ of $M_{n,k}^k(\mc F)$ there is $R\in \mc A\cap M_{n,m}^m(\mc F)$ such that $\mu$ is a factor of $c$ in $R\cdot M_{m,k}^k(\mbb F)$. 
\end{enumerate}  
We call $X$ the {\em set of colors} of $\mu$, denoted  by $X_\mu$.
\end{definition}
It follows that $(\tau,\mc E)$ is a $k$-Ramsey factor, and it is the minimal one in the following precise sense. 

\begin{proposition}\label{8ret4433344}
Suppose that $(\mu, \mc A)$, $(\nu, \mc B)$ are $k$-Ramsey factors. 
\begin{enumerate}[a)]
\item $\#X_\mu\ge \#\mr{GL}(\mbb F^k)=\prod_{j=0}^{k-1} ((\#\mbb F)^k -(\#\mbb F)^j)$.
\item    If $\mc A\con \mc B$, then there is a surjection $\theta: X_\mu\to X_\nu$ such that $\mu \circ \theta=\nu$. 
 \item   If $\mc A= \mc B$, then there is a bijection $\theta: X_\mu\to X_\nu$ such that $\mu \circ \theta=\nu$.
\end{enumerate}

\end{proposition}
\begin{proof}
 {\it a)}: In fact, we prove that if $(\mu,\mc A)$ satisfies {c)}  of Definition \ref{lijeriowijorewidse}, then $\#X\ge \#\mr{GL}(\mbb F^k)=\prod_{j=0}^{k-1} ((\#\mbb F)^k -(\#\mbb F)^j)$. Find $n\ge k$,  $\theta: X\to \mr{GL}(\mbb F^k)$ and $R\in \mc A\cap  $ such that $\tau(R\cdot A)= \theta (\mu(R \cdot A))$ for every $A\in M_{k,k}^k(\mbb F)$.  It is easy to see that $\tau: R\cdot M_{k,k}^k  (\mbb F)\to \mr{GL}(\mbb F^k)$ is surjective, hence $\theta$ is also surjective.
 {\it b)}: With same strategy  one easily proves {\it b)}.    {\it c)}: From {\it b)} we have that $\# X_\mu=\#X_\nu$, and $\theta$ in {\it b)} must be a bijection. 
\end{proof}

\section{Matrices and Grassmannians over $\mathbb{R},\mathbb{C}$}\label{Ramsey_R_C}
We present    factorization results of compact colorings of matrices and Grassmannians over the fields $\mbb F=\R,\C$.    There are several such results, depending on the chosen metric on the objects we color. These factorizations are approximate, because, as we deal with infinite fields, it is easily seen that the exact ones are not true; on the other hand, they apply to arbitrary colorings given by Lipschitz mappings with values in a compact metric space.    Given $\al,\be\in \N\cup \{\infty\}$, the collection of matrices $M_{\al,\be}(\mathbb{F})$ can be naturally turned into a metric space by fixing two norms $\mtt m$ and $\mtt n$ on $\mbb F^\al$ and $\mbb F^\be$, respectively, and identifying a matrix $A\in M_{\al,\be}$ with the linear operator $T_A:\mbb F^\be\to \mbb F^\al$,  $T_A(x):= A\cdot x$, $x$ as column vector (i.e., a $\be\times 1$-matrix). This allows to define the norm $\nrm[A]_{\mtt m,\mtt n}:=\nrm[T_A]_{(\mbb F^\be, \mtt m),(\mbb F^\al,\mtt n)}$, and the corresponding   distance $d_{\mtt m,\mtt n}(A,B):=\nrm[A-B]_{\mtt m,\mtt n}=\nrm[T_A-T_B]_{(\mbb F^\be, \mtt m),(\mbb F^\al,\mtt n)}$. Also, in this way each    full rank $\al\times k$-matrix $A$ defines a norm $\nu(A)$ on $\mbb F^k$, $\nu(A)(x):= \mtt n(A\cdot x)$. 
 When $\mtt m$ is a norm on $\mbb F^\infty$,  by identifying  each $\mbb F^k$ with $\langle u_j\rangle_{j<k}$,   let $\mtt m_k$ be the norm on $\mbb F^k$, $\mtt m_k( (a_j)_{j<k}):=\mtt m(\sum_{j<k} a_j u_j)$. When there is no possible misunderstanding,  we will write $d_\mtt m$ to denote $ d_{\mtt m_\be, \mtt m_\al}$.

 Recall that in general, given two normed spaces $X=(V,\mtt m)$  and $Y=(W,\mtt n)$,  $\mc L(X,Y)$ denotes the space of continuous (equivalently bounded) linear operators from $X$ to $Y$, that is again a normed space by considering the norm $\nrm[T]:=\sup_{x\in \ball(X)} \mtt n(T(x))$, where $\mathrm{Ball}(X)=\conj{x\in X}{\mtt m(x)\le 1}$ denotes the unit ball of $X$.   Let  $\mc L^k(X,Y)$ is the set of those operators of rank $k$.  Since when $V$ is finite dimensional every linear mapping from $V$ to $W$ is automatically continuous, in this case, we will use also $\mc L(V,W)$ and $\mc L^k(V,W)$, to denote the collection of linear mappings from $V$ to $W$, and those of rank $k$, respectively. By an {\em isometric embedding} we mean a linear mapping $T: X\to Y$  such that $\mtt n(T(x))=\mtt m(x)$ for every $x\in X$. The space of these operators is denoted by $\Emb(X,Y)$.

Of particular importance will be  the $p$-norms. 
     Recall that for every $1\le p\le \infty$, $\ell_p^n$ is the normed space $(\mbb F^n, \nrm_p)$, where $\nrm[(a_j)_{j<n}]_p:=(\sum_{j<n} |a_j|^p)^{1/p}$ for $p<\infty$ and $\nrm[(a_j)_{j<n}]_\infty:=\max_{j<n}|a_j|$.  Similarly one defines  the $p$-norms on $\mbb F^{\infty}$, that we denote as $\ell_p^\infty:=(\mbb F^\infty, \nrm_p)$, and their completions are usually denoted by $\ell_p$, for $p<\infty$ and by $c_0$, when $p=\infty$.

Roughly speaking, our factorization theorem  for full rank  matrices  (Theorem \ref{factor_p_full_matrices}) states that every coloring of  such matrices,  endowed with the $p$-metrics for $p\in  [1,+\infty]\setminus 2(\N+2)$   is ``approximately determined'' by the corresponding $\nu$ described above.  

 Similarly, once a norm $\mtt m$ is fixed in $\mbb F^\al$, $\mr{Gr}(k,\mbb F^\al)$ turns into a metric space by considering a Hausdorff metric, and  each $k$-dimensional subspace $V$ of $\mbb F^\al$ determines a member of the \emph{Banach-Mazur} compactum $\mc B_k$, that is,   the isometry class $\tau_\mtt m(V)$ of all $k$-dimensional normed spaces isometric to $(V,\mtt m)$. We prove that when choosing   $p$-norms on each $\mbb F^n$ for $n$ large enough, any coloring of the $k$-Grassmannians of $\mbb F^n$ is approximately determined  by $\tau_\mtt m$ on some $\mr{Gr}(V,k)$.
We introduce a more appropriate terminology, in particular we extend the type of colorings to work with. A \emph{metric
coloring} of a pseudo-metric space $M$ is a $1$-Lipschitz map $c$  from $M$ to a
metric space $(K,d_{K})$. We will say that  $c$  is a $K$-coloring. Following \cite%
{melleray_extremely_2014}, a \emph{continuous coloring} is a metric coloring
whose target space is the closed unit interval $\left[ 0,1\right] $. A \emph{%
compact coloring }is a metric coloring whose target space is a compact
metric space. For a subset $X$ of a   metric space $\left(
K,d_{K}\right) $ and $\varepsilon >0$, the $\varepsilon $-fattening $%
X_{\varepsilon }=\conj{p\in K}{\text{ there is some $q\in X$ with $d(p,q)\le \vep$}}$.
% is  the set of points to  a distance at most $%
%\varepsilon $ from some point of $X$. 

The \emph{oscillation }$\mathrm{osc}(c\rest F)$ of a compact coloring $%
c:M\rightarrow (K,d_{K})$ on a subset $F$ of $M$ is the supremum of $%
d_{K}(c(y),c(y^{\prime }))$ where $y,y^{\prime }$ range within $F$. When $%
\mathrm{osc}(c\rest F)\leq \varepsilon $  we also say that $c$ $%
\varepsilon $-\emph{stabilizes }on $F$, or that $F$ is $\varepsilon $-\emph{%
monochromatic }for $c$. A \emph{finite} coloring of $M$ is a function from $%
c $ from $M$ to a finite set $X$; in the particular case when the target space is a natural number $%
r$ (identified with the set $\left\{ 0,1,\ldots ,r-1\right\} $ of its
predecessors), we will say that $c$ is an $r$-coloring. Given a finite coloring $c:M\to X$ and $\vep\ge 0$, we say that a subset $F$ of $M$
is $\vep$-monochromatic for $c$, or that $c$ $\vep$-stabilizes on $F$,  if there exists some $x\in X$ such that $F$ is included in the $\vep$-neighborhood $(c^{-1}(x))_\vep$ of $c^{-1}(x)$. When $\vep=0$ we will omit the use of the prefix ``$0$-''.
%such that for every $p\in F$ there is $q\in E$ such that $c(q)=x$ and $%
%d_{M}(p,q)\leq \varepsilon $. 
%If $F$ is $\varepsilon $-monochromatic, then
%we also say that $c$ $\varepsilon $-stabilizes on $F$.

\defi[Approximate factors]
Let $(M,d_M)$, $(N,d_N)$ and $(P,d_P)$ be metric spaces, $\vep> 0$, and $c:(M,d_M)\to (N,d_N)$  and $\pi: (M,d_M)\to (P,d_P)$ be metric colorings, i.e., 1-Lipschitz maps.  We say that $\pi$  is an  \emph{$\vep$-approximate factor} (or simply \emph{$\vep$-factor}) of $c$  if there is some   coloring   $\widetilde c: (P,d_P)\to (N,d_N)$ such that
$$\sup_{x\in M} d_N( c(x), \widetilde c (\pi(x)))\le \vep.$$
That is, ``up to $\vep$" $c= \widetilde c \circ \pi$. Given $M_0\con M$ we say that $\pi$ is an \emph{$\vep$-factor of $c$ in $M_0$} if $\pi\upharpoonright _{M_0}:M_0\to P$ is an $\vep$-factor of  $c\upharpoonright {M_0}$, i.e., there is some coloring $\widetilde{c}: P\to N$ such that $\sup_{x\in M_0} d_N(c(x),\widetilde{c}(\pi(x)))\le \vep$.
\fdefi
\subsection{The statements. Ramsey factors}
As discussed above, given norms $\mtt m,\mtt n$ on $\mathbb{F}^m$ and $\mathbb{F}^n$ respectively, we regard  $M_{n,m}$ as a metric space by considering a   $n\times m$-matrix $A$ as the particular representation of a linear operator $T_A$ in the unit bases  of suitable normed spaces $(\mbb F^m,\mtt m)$  and $(\mbb F^n,\mtt n)$, and then by considering the corresponding operator norm.

 \subsubsection{Full rank matrices}
   Given  a vector space $V$, let $\mc  N_V$ be the   set of all norms on $V$, endowed with the topology of pointwise convergence.   When $\dim V<\infty$,   a compatible metric on $\mc N_V$ is   
   \[\omega (\mtt m,\mtt n):=\log \max \{\Vert \mathrm{Id}\Vert _{(V,\mtt m),(V,\mtt n)},\Vert \mathrm{Id}\Vert _{(V,\mtt n),(V,\mtt m)}\},\]
    that will be called the \emph{intrinsic metric} on $\mc N_V$.  It is easy to see that the metric space $(\mc N_V, \om)$ has the {\em Heine-Borel} property, that is, every closed and bounded set of it is compact. In particular, each closed $\om$-ball is compact.          
    Given a normed space $E=(W,\nrm)$,  let  $\mc N_{V}(E)$ be the collection of  norms $\mtt m$ on $V$ such that there exists a linear isometry $T: (V,\mtt m)\to E$. In general, $\mc N_V(E)$ is not  closed  in  $\mc N_V$, although in some  natural cases is. We will write $\mc N_\al$ to denote $\mc N_{\mbb F^\al}$
  \begin{definition}
Suppose that $V$ is finite dimensional, $E=(W,\nrm)$ a normed space.   Let 
\[\nu_{V,E}: \mc L^{\dim V}(V,W) \to\mc N_V(E)\]  be the mapping that assigns to a 1-1 linear mapping $T:V\to W$  the norm $\nu_{V,E}(T)$ on $V$, defined  by  $(\nu_{V,E}(T))(x):= \nrm[T(x)]$, that is, the norm on $V$   that makes $T$ an isometric embedding. With a slight abuse of notation, we also write $\nu_{k, (\mbb F^\al,\nrm)}$  to denote the mapping   $A\in M_{\al,k}^k\mapsto \nu_{\mbb F^k,(\mbb F^\al,\nrm)}(T_A)$ that assigns to  a such matrix $A$ the norm  defined for each $x\in \mbb F^k$ by $(\nu_{k,E}(A))(x):= \nrm[A\cdot x]$.
 \end{definition}   
Given a finite dimensional normed space $X=(X,\nrm_X)$  and a normed space $E=(V,\nrm_E)$,  we define on $\mc N_X(E)\times\mc N_X(E)$  the {\em $E$-extrinsic function} 
$$\partial_{X,E}(\mathtt{m},\mathtt{n}):= \inf\conj{\nrm[T-U]_{X,E} }{T\in \Emb((X,\mtt m), E), \, U\in \Emb((X,\mtt{n}),E)}.$$ 
So, $\partial_{X,E}(\mtt m,\mtt n)$ computes the minimal distance $d_{X,E}(T,U)$ between possible representations of $\mtt m$ and $\mtt n$, $\nu_{X,E}(T)=\mtt m$, $\nu_{X,E}(U)=\mtt n$. 
In general $\partial_{X,E}$ is not a compatible metric. 
Note that $\partial_{X,E}$  is a metric  when $\partial_{X,E}$ satisfies the triangle inequality.
 The following is easy to prove.
\begin{proposition}
 If  $\partial_{X,E}$ is  compatible,    $\nu_{X,E}:(\mc L^{\dim X}(X,E), d_{X,E})\to  (\mc N_X(E), \partial_{X,E}) $ is 1-Lipschitz.   \qed
\end{proposition}

Recall that given a linear operator $T:X\to Y$ between normed spaces $X$ and $Y$, 
\[\nrm[T]= \min \conj{\la\ge 0}{ T (\ball(X))\con \la \cdot \ball(Y)},\]
and when $X$ is finite dimensional, let
\[ \nrm[T^{-1}]= \min \conj{\la\ge 0}{  \ball(T(X))\con \la \cdot T(\ball(X))}.\]
When $T$ is 1-1, $\nrm[T^{-1}]=\nrm[U]$, where $U: TX\to X$ is the inverse operator of $T$.  
Given $\al,\be \in \N\cup \{\infty\}$, and a norm $\mtt m\in \mc N_\infty$,  let $M_{\al,\be}^k(\mtt m;\la)$ be the collection of  matrices in $M_{\al,\be}^k$  such that the corresponding linear operator $T_A:(\mbb F^\be,\mtt m)\to (\mbb F^\al,\mtt m)$  satisfies that   $\nrm[T_A],\nrm[T_A^{-1}]\le \la$.  
$$\la^{-1}\mr{Ball}(\im T_A)\con T_A(\mr{Ball}(\mbb F^\infty,\mtt m))\con \la \mr{Ball}(\mbb F^\infty,\mtt m) .$$
 Let  also  $M_{\al,k}^k(\mtt m;\sma\la)=\bigcup_{1\le \mu <\la}M_{\al,k}^k(\mtt m,\mu)$, that is, the  matrices $A\in M_{\al,k}^k$ such $\nrm[T_A], \nrm[T_A^{-1}]<\la$.    
Notice that the boundary $M_{\al,k}^k(\mtt m;1)$ is the collection of matrices $A$ defining isometric embeddings $T_A:(\mbb F^k,\mtt m)\to (\mbb F^\infty,\mtt m)$, and it will be denoted by $\mc E_{\al,k}(\mtt m)$, and $\mc E(\mtt m):=\bigcup_{n\ge m}\mc E_{n,m}(\mtt m)$.  The following is easy to prove, and  highlights the interest of this collection. 
\prop\label{fdsdijfidsewrw32}
Let $R\in M_{\al,m}^m$. 
\begin{enumerate}[a)]
\item The multiplication   by $R$ operator $\mu_A: (M_{m,k}^k, d_{\mtt m})\to (M_{\al,k}^k,d_{\mtt m})$,  $A\mapsto \mu_R(A):=R\cdot A $ defines an isometry  if and only if $R\in \mc E_{\al,m}(\mtt m)$.
 \item If $R\in M_{\al,m}(\mtt m)$, then $\nu_{k,(\mbb F^\al,\mtt m)}\circ \mu_R= \nu_{k,(\mbb F^m,\mtt m)}$. \qed
\end{enumerate} 
\fprop
\begin{proof}
 {\it a)}:  Suppose that $X$ is a normed space of finite dimension $k$, $Y,Z$ normed spaces and suppose that $T\in \mc L(Y,Z)$ is such that the composition operator $U\in \mc L^k(X,Y)\mapsto T\circ U\in \mc L(X,Z)$ is an isometry with respect to the norm metrics. Let us prove that $T$ must be an isometry. Fix a non-zero vector $y\in Y$. Let $(x_j)_{j<k}$ be an  {\em Auerbach basis} of $X$, i.e., a basis consisting of normalized vectors such that its biorthogonal sequence $(x_j^*)_{j<k}$ is also normalized   (see \cite[Chapter 4, Theorem 13]{bollobas_linear_1999}). Let $(y_j)_{j<k}$ be a linearly independent sequence in $Y$ with $y_0=y$. For each $n\ge 1$, let $T_n(x)= x_0^*(x) y + (1/n)\sum_{j=1}^{k-1} x_j^*(x) y_j$. It is clear that $T_n$, $(1/n) T_n$ are  $1-1$. Then,  $\nrm[U\circ T_n -U\circ (1/n)\cdot T_n]-\nrm[U\circ T_n]\to_n 0$, and $\nrm[U\circ T_n -U\circ (1/n)\cdot T_n]-\nrm[T_n]=\nrm[ T_n - (1/n)\cdot T_n]-\nrm[T_n]\to_n 0$, hence $\nrm[U\circ T_n]-\nrm[T_n]\to_n 0$.   It follows that $\nrm[T_n]\to_n \nrm[x_0^*]^* \nrm[y]$ and similarly  $\nrm[U\circ T_n]\to_n \nrm[x_0^*]^* \nrm[U(y)]$. Since $\nrm[x_0^*]^*=1$, we obtain $\nrm[U(y)]=\nrm[y]$. {\it b)} is trivial.
\end{proof}

Given a normed space $E=(\mbb F^\infty, \nrm_E)$, let  $\mc N_k(E;\la)$  ($\mc N_k(E;\sma\la)$) be the closed (resp. open) ball of $\mc N_k(E)$ with respect to the intrinsic metric $\om$ centered on the norm $\nrm_E$  in $\mbb F^k$  and with radius $\la$, i.e., $\mc N_k(E; \la)=\conj{\mtt n\in \mc N_k(E)}{\om(\mtt n,\nrm_E \rest\langle u_j\rangle_{j<k})\le \log\la}$, similarly  for $\mc N_k(E; \sma\la)$.  
\begin{proposition}
 \begin{enumerate}[a)]
\item 
 $\mc N_k(E;\la)= \nu_{k,E}(M_{\al,k}^k(\nrm_E;\la))$ and    $\mc N_k(E;\sma\la)= \nu_{k,E}(M_{\al,k}^k(\nrm_E;\sma\la))$.
\item If $\partial_{X,E}$ and $\om$ are uniformly equivalent on $\om$-bounded subsets of $\mc N_X(E)$, then every $\om$-bounded set is $\partial_{X,E}$-totally bounded,  thus, the $\partial_{(\mbb F^k,\mtt m), E}$-completion of $\mc N_k(E;\la)$ and $\mc N_k(E;<\la)$ are compact.  
 \end{enumerate}
 
\end{proposition}  
\begin{proof}
{\it b)}: Suppose that $A\con \mc N_X(E)$ is $\om$-bounded. Since $(\mc N_X,\om)$ has the Heine-Borel property, $A$ is $\om$-totally bounded.  $\partial_{X,E}$ is uniformly equivalent to $\om$ on $A$, so $A$ is $\partial_{X,E}$-totally bounded.
\end{proof}

\begin{definition}[Ramsey factors for full-rank matrices]\label{io4jrio4rfer3w}
Let  $\mtt m$ be a metric on $\mbb F^\infty$, set  $E_\al:=(\mbb F^\al,\mtt m)$ for every $\al\le \infty$. We say that $\mtt m$  {\em produces Ramsey factors} for colorings of full rank matrices, when 
  \begin{enumerate}[i)]
 \item     $\partial_{E_k,E_\infty}$ is a compatible metric on $\mc N_k(E_\infty)$ uniformly equivalent to $\om$ on $\om$-bounded sets.   
 \item     Given $k,m\in \mathbb{N}$,   $\varepsilon>0$,     $ \la>1$  and a  compact metric space $(K,d_{K})$  
there is $n\in \N$ such that  for every  $K$-coloring $c$ of    $(M^k_{n, k}({\mtt m;\la}),d_{\mtt m})$ there is $R\in \mathcal{E}_{n, m}(\mtt m)$ such that the restriction $\nu_{\mbb F^k, E_\infty}: M_{\infty,k}^k(\mtt m; \sma\la)\to \mc N_k(E_\infty;\sma\la)$ is an $\vep$-factor of $c$ in  $R\cdot M_{m,k}^k(\mtt m;\sma\la)$; that is,  there is a coloring $\widetilde{c}: (\mc N_k(E_\infty; \sma\la), \partial_{E_m,E_\infty})\to (K,d_K)$  such that  $d_K(c(R\cdot A), \widetilde{c}(\nu_p(A)))\le \vep$ for every $A\in M_{m,k}^k(\mtt m;\sma\la)$.
   
 \end{enumerate}   
\end{definition}  

\begin{theorem}[$p$-Factorization of colorings of full rank matrices  over $\mathbb R$, $\C$]\label{factor_p_full_matrices}
 For $1\le p\le \infty$, $p\notin 2\N+4$, the $p$-norm $\nrm_p\in \mc N_\infty$    {\em produces Ramsey factors} for colorings of full rank matrices.
  \end{theorem}

 \subsubsection{Grassmannians}
Given a  normed space $E=(V,\nrm_E)$ the  $k$-Grassmannian $\gr(k,V)$ of $V$ is  naturally a  topological space,  as it can be identified with the corresponding topological quotient of $E^k$ by the relation $(x_j)_{j<k}\sim (y_j)_{j<k}$
  iff $\langle x_j\rangle_{j<k}=\langle y_j\rangle_{j<k}$.  If in addition $E$ is separable, this turns $\mr{Gr}(k,E):=\mr{Gr}(k,V)$ into a polish space.  A natural compatible metric is  the \emph{gap (or opening) metric} (see \cite{kalton_distances_2008}), $\La_{k,E}(U,W)$ defined as    the Hausdorff distance, with respect to the norm metric in $E$, between the unit balls  $\mr{Ball}(U,\mtt m)$ and  $\mr{Ball}(W,\mtt m)$, that is,
$$\La_{k,E}(U,W ):= \max \{ \max_{u\in \ball(U,\nrm_E)} \min_{w\in \ball(W,\nrm_E)} \nrm[u-w]_E,    \max_{w\in \ball(W,\nrm_E)}\min_{u\in \ball(U,\nrm_E)} \nrm[w-u]_E\}.$$
Let $\mr{GL}(V)\acts \mc N_V$ be the canonical action $ (\De\cdot \mtt m  )(v):= \mtt m( \De^{-1}( x))$ for every $x\in V$. Notice that the intrinsic metric $\om$ is invariant under this action.  Let  $\mc B_V:=\mc N_V \quo \mr{GL}(V)$ be the quotient space. Since $\om$ is invariant under the action and $\mc N_V$ has the Heine-Borel property, so is $\mc B_V$ with its quotient metric.  Observe that given a norm $\mtt m$ on $V$ of dimension $k$ there is a linear transformation $\De$ such that $(\De(u_j))_{j<k}$ is an Auerbach basis of $(V,\mtt m)$, i.e., a normalized sequence such that $\mtt m(\sum_{j<k} a_j \De(u_j))\ge \max_{j<k} |a_j|$ for every sequence of scalars $(a_j)_{j<k}$. This implies that given two norms $\mtt m,\mtt n\in \mc N_V$ there is a linear transformation $\De$ such that $\om(\De \cdot \mtt m, \mtt n)\le \log k$, and consequently the diameter of $\mc B_V$ is at most $\log k$, hence it is compact, called the {\em Banach-Mazur compactum}. The quotient metric corresponding to $\om$  is 2-Lipschitz equivalent to the well-known {\em Banach-Mazur metric}  
$$d_\mr{BM}(\mtt m,\mtt n):=\log \inf_{\De\in\mathrm{GL}(V)} \nrm[\De]_{(V,\mtt m),(V,\mtt n)}\cdot \nrm[\De^{-1}]_{(V,\mtt n),(V,\mtt m)}.$$ 
 Let $\mc B_V(E)$ denote the Banach-Mazur classes corresponding to norms in $\mc N_V(E)$, or, in other words, the isometric types of finite dimensional subspaces of $E$ of the same dimension than $V$. We write  $\mc B_k$, $\mc B_k(E)$ and $\gr(k,E)$  to denote $\mc B_{\mbb F^k}$, $\mc B_{\mbb F^k}(E)$ and   $\gr(k,V)$, respectively. 
\begin{definition}
Let  $\tau_{k,E}:\mr{Gr}(k,E)\to \mathcal B_k(E)$   be the mapping that assigns to each  $k$-dimensional  normed subspace $W$ of $E$  the isometric type of $(W,\nrm_E)$.
\end{definition}  
 In other words,   for    $W\in
\mr{Gr}(k,E)$, $\tau_{k,E}(W)=[\nu_{k,E}(T)]_\mr{BM}$ for some  1-1 linear  function $T: \mbb F^k\to W$ such that   $\im T= W$.  
We  define   the \emph{$E$-Kadets} mapping
 $\ga_{k,E}$ on $\mc B_{k}(E)\times \mc B_{k}(E)$ by 
 \[\ga_{k,E}([\mtt m],[\mtt n]):=\inf\conj{\La_{k,E}(U,W)}{\text{$U,W\in \mr{Gr}(k, E)$,  
  $(U,\nrm_E)\equiv (\mbb F^k,\mtt m)$, $(W,\nrm_E)\equiv (\mbb F^k,\mtt n)$}}.\]
  \begin{definition}
 $\ga_{k,E}$ is the {\em $E$-Kadets  metric} when it is a compatible metric on $\mc B_{k}(E)$. 
  \end{definition}  
In the literature the Kadets metric $\ga$  corresponds
 to the metric $\ga_{C[0,1]}$  for Grassmannians of the universal space of continuous functions on the unit interval $C[0,1]$ (see \cite{kalton_distances_2008}).
  
 \begin{proposition}
If 
 $\ga_{k,E}$ is  Kadets,   $\tau_{k,E}: (\mr{Gr}(k,E), \La_E)\to (\mc B_k(E), \ga_{k,E})$ is 1-Lipschitz.  \qed
\end{proposition}

Given $E=(\mbb F^\al,\mtt m)$, we write $\mr{Gr}_{\mtt m}(k, \mbb F^\al)$  to denote the set  of $k$-dimensional subspaces $W$ of $\mbb F^\al$ so that $(W,\mtt m)$ is  isometric to $(\mbb F^k,\mtt m)$, i.e., $\tau_{k,E}(W)=[ \mtt m \rest\langle u_j\rangle_{j<k}]$. 
The next explains the interest of   $\mr{Gr}_{\mtt m}(k, \mbb F^\al)$ and it is proved similarly  to Proposition  \ref{fdsdijfidsewrw32}.
\begin{proposition}
Fix $k\le m$ and $W\in \gr(m, E)$.
\begin{enumerate}[a)]
\item An invertible linear  operator $\theta: W\to \mbb F^m$ defines an isometry $\Theta: (\gr(k, W), \La_{k,E})\to (\gr(k, \mbb F^m), \La_{k,E})$, $V\mapsto\theta(V)$,  if and only if $\theta:(W,\mtt m)\to (\mbb F^m,\mtt m)$ is an isometry.
\item  $\tau_{k,(W,\mtt m)}=\tau_{k,E}\rest  \mr{Gr}(k, W)$.  
\end{enumerate}  
 \end{proposition}  
 
\begin{definition}[Ramsey factors for Grassmannians] A norm $\mtt m$ on $\mbb F^\infty$, $E=(\mbb F^\infty,\mtt m)$,  {\em produces Ramsey factors} for colorings of Grassmannians when
\begin{enumerate}[i)]
\item  $\ga_{k, E}$ is a compatible metric.
\item  For every  $k\in \N$, $\vep>0$, and every   compact metric  $(K,d_{K})$     there is $n$ such that for every $c:(\mr{Gr}(k, \mbb F^n),\La_{k,E})\to (K,d_K)$ there is
$V\in \mr{Gr}_{\mtt m}(m,\mbb F^n)$    such that  $\tau_{k, E}$ is an $\vep$-factor of $c$ in $\mr{Gr}(k,V)$.  

\end{enumerate}

\end{definition}  
%We have the following.
\begin{theorem}[Factorization  of Grassmannians over $\R$, $\C$]\label{factor_p_grass}  For $1\le p\le \infty$, $p\notin 2\N+4$, the $p$-norm $\nrm_p\in \mc N_\infty$    {\em produces Ramsey factors} for colorings of Grassmannians.
  \end{theorem}
Geometrically, the previous result states that restrictions of compact colorings of Grassmannians depend on shapes of their unit balls.  
 The following statement can be considered as a version of the Graham-Leeb-Rothschild Theorem for the fields $\mbb R, \mbb C$.
\cor[Graham-Leeb-Rothschild for $\mbb R$, $\mbb C$]\label{oi43roijre4}
For every $1\le p\le \infty$, $p\notin 2\N+4$, every $k,m\in \N$, every $\vep>0$ and every compact metric space $(K,d_K)$   there is $n$ such that  every coloring  $c:(\mr{Gr}(k, \mbb F^n), \La_{k,\ell_p^\infty})\to (K,d_K)$   $\vep$-stabilizes on $ \mr{Gr}(k,W)$ for some $W\in \mr{Gr}(m,\mbb F^n)$.
\fcor
\prue
This is direct consequence of Theorem \ref{factor_p_grass} and the facts that for large $r$ the space $\ell_p^r$ contains almost isometric copies of $\ell_2 ^m$ and that $\mc B_k(\ell_2^m)$ consists of a point.
\fprue
\nota\label{lkejroiwerjweirwe}
Recall that Dvoretzky's Theorem asserts that any finite-dimensional normed space $X$ of dimension $r$ contains almost isometric copies of $\ell_2^m$  with $m$ proportional to $\log(r)$
  (see \cite[Theorem 12.3.6]{albiac_topics_2006}). This means that Corollary \ref{oi43roijre4} remains true for every norm $\mtt m$ on $\mbb F^\infty$.
  
\fnota
 
 \subsubsection{Square matrices}
Given a vector space $V$, let $V^*$ be its (algebraic) dual, the vector space  of linear functions $f:V\to \mbb F$; if in addition $X=(V,\mtt m)$ is a normed space, $X^*$ will denote   the (normed) dual space $\mc L(X,(\mbb F,|\cdot|))$, that is,  the vector space of continuous linear functionals $f:V\to \mbb F$ endowed with the {\em dual norm} $\mtt m^*(f):=\sup_{\mtt m(x)\le 1} |f(x)|$.    Let $\mr{GL}(V)\acts \mc N_{V^*}$ be the canonical action  $(\De\cdot \mtt n)(f):= \mtt n(\De^*(f))$ for $f\in V^*$. Observe that the dual mapping ${\cdot}^*: (\mc N_V,\om)\to( \mc N_{V^*},\om)$ is a $\mr{GL}(V)$-equivariant isometry, that is,   given $\mtt m\in \mc N_V$, $(\De\cdot \mtt m)^*=\De\cdot \mtt m^*$. Let 
 $\mr{GL}(V)\acts \mc N_V \times \mc N_{V^*}$ be the action  $\De \cdot (\mtt m,\mtt n):= (\De \cdot \mtt m, \De\cdot \mtt n)$ and let  $\mc D_V$  be the quotient space $(\mc N_V\times \mc N_{V^*})\quo \mr{GL}(V)$. With the compatible metric $\om_2((\mtt m, \mtt n), (\mtt p,\mtt q)):= \om(\mtt m, \mtt p)+\om(\mtt n,\mtt q)$ the product $\mc N_V\times \mc N_{V^*}$ has the Heine-Borel property, so with the corresponding quotient metric $\wt \om_2$,   $\mc D_V$ also has this property.    Given a normed space $E$, let   $\mc D_V(E):=(\mc N_V(E)\times \mc N_{V^*}(E))\quo \mr{GL}(V)$. Its orbits will be  denoted by $[ {\mtt{\mbf m}}]=[(\mtt m_0,\mtt m_1)]$. 
 In the next    $E:=(\mbb F^\al, \nrm_E)$ and $\mc D_k$, $\mc D_k(E)$ denote  $\mc D_{\mbb F^k}$ and $\mc D_{\mbb F^k}(E)$, respectively.   
 \begin{definition} 
Let $\nu^2_{k,E}: M_{\al}^k  \to \mc D_k(E)$ be the function that assigns to an  $\al$-squared  matrix $A$ of rank $k$, the class   $\mr{GL}(\mbb F^k)$-orbit  of the pair $(\nu_{p}(B), \nu_{p}(C))$ for $B,C\in M_{\al,k}^k$   with $A=B\cdot C^*$. 
\end{definition}  
The fact that $\nu_{k,E}^2$ is well defined follows from the full-rank factorization of matrices. 
 \begin{proposition}
$[(\nu_{\mbb F^k, E}(T_0), \nu_{(\mbb F^k)^*,E}(T_1))]=\{(\nu_{\mbb F^k, E}(U_0), \nu_{(\mbb F^k)^*,E}(U_1)) : U_0 \circ U_1^*= T_0\circ T_1^*\}$. 
\end{proposition}  
\begin{proof}
 If $ U_0,U_1: \mbb F^k \to E$ are linear operators of rank $k$  then  $T_0 \circ T_1^*= U_0\circ U_1^*$ if and only if there is some $\De\in \mr{GL}(\mbb F^k)$ such that $U_0= T_0\circ \De^{-1}$ and $U_1= T_1\circ \De^*$. 
 \end{proof}  
  We define on   $\mc D_k(E)\times \mc D_k(E)$ the function $\mk d_{k,E}$
  $$\mk d_{k,E}([\mtt{\mbf m}], [\mtt{\mbf n}])):=\inf_{T,U}\nrm[T-U]_{E^*,E}$$      where the infimum is over bounded linear mappings $T,U: E^*\to E$ of rank $k$ admitting decompositions $T=T_0 \circ T_1^*$ and $U=U_0\circ U_1^*$  with  $T_0,U_0:\mbb F^k \to E$ and  $T_1,U_1:\mbb F^k \to E$ of rank $k$ for $j=0,1$ and  such that $(\nu_{\mbb F^k,E}(T_0),\nu_{(\mbb F^k)^*,E}(T_1))\in [\mtt{\mbf m}]$ and $(\nu_{\mbb F^k,E}(U_0),\nu_{(\mbb F^k)^*,E}(U_1))\in [\mtt{\mbf n}]$. In the previous, we are identifying canonically $(\mbb F^k)^*$ with $\mbb F^k$.   The following is easy to prove.
 \begin{proposition}
 If $\mk d_{k,E}$ is  compatible,  $\nu_{k,E}^2: (M_{\al}^k, d_{E^*,E})\to (\mc D_k(E), \mk d_{k,E})$
 is 1-Lipschitz.  \qed
\end{proposition}

 Given a norm  $\mtt m\in \mc N_{\mbb F^{\al}}$,  let $M_{\al}^{k}(\mtt m;\la)$ be the collection of    $A\in M_{\al}^k$  such that $\nrm[T_A],\nrm[T_A^{-1}]\le \la$.   The next has a    similar proof to that of Proposition   \ref{fdsdijfidsewrw32}, so we leave the details to the reader. 
 \prop
Let $L\in M_{\al,m}^k$ and $R\in M_{m,\al}^k$.
\begin{enumerate}[a)]
\item The multiplication by $L$ and $R$ function $\mu_{L,R}:(M_{m}^k, d_{E^*,E})\to (M_{\al}^k, d_{E^*,E})$, $A\in M_{m}^k\mapsto \mu_{L,R}(A):=L\cdot A \cdot R \in M_{\al}^k$ is  an isometry if and only if $L, R^*\in \mc E_{\al,m}(\nrm_E)$.
 \item If  $L, R^*\in \mc E_{n,m}(\nrm_E)$, then $\nu_{k,(\mbb F^\al, \nrm_E)}^2\circ \mu_{L,R}= \nu_{k,(\mbb F^m, \nrm_E)}^2$. \qed
 \end{enumerate} 
\fprop

Given $\la\ge 1$, let $\mc D_k(\la):=\conj{[(\mtt m,\mtt n)] \in \mc D_k}{ \om(\mtt m^*,\mtt n)\le \la }$,   $\mc D_k(\sma \la):=\conj{[(\mtt m,\mtt n)] \in \mc D_k}{ \om(\mtt m^*,\mtt n)< \la }$,  and let $\mc D_k(E;\la):=\mc D_k(\la)\cap \mc D_k(E)$, and $\mc D_k(E;\sma\la):=\mc D_k(\sma\la)\cap \mc D_k(E)$.  
\prop
\begin{enumerate}[a)] 
\item $\mc D_k(\la)$ and $\mc D_k(\sma\la)$ are well defined.
\item $\mc D_k(\la)$ is compact. 
\item  $\mc D_k(E;\la)=\nu_{k,E}^2( M_\al^k(\nrm_E;\la))$ and $\mc D_k(E;\sma\la)=\nu_{k,E}^2( M_\al^k(\nrm_E;\sma\la))$.
  
\end{enumerate}  
\fprop 
\begin{proof} {\it a)} follows from the fact that given $\De\in \gl(\mbb F^k)$, 
 $\om((\De\cdot \mtt m)^*, \De\cdot \mtt n)=\om(\De\cdot \mtt m^*, \De\cdot \mtt n)=  \om( \mtt m^*,  \mtt n)$. {\it b)}: It is clear that $\mc D_k(\la)$ is closed, so by the Heine-Borel property of $(\mc D_k,\wt\om_2)$, we just have to prove that $\mc D_k(\la)$ is $\wt\om_2$-bounded:       Fix $[(\mtt m,\mtt n)],[(\mtt p,\mtt q)]\in \mc D_k(\la)$, let $\De\in \gl(\mbb F^k)$ be such that $\om(\De\cdot \mtt m,\mtt p)\le \la$. Then it follows that 
$\om(\De\cdot \mtt n, \mtt q)\le \om(\De \cdot \mtt n, \De\cdot  \mtt m^*) +\om( \De\cdot  \mtt m^*,  \mtt p^*)+\om( \mtt p^*,  \mtt q) \le 2\la +k$.  {\it c)} will be proved in   Lemma \ref{jiowioerio43r446576}.
\end{proof}

\begin{definition} A norm $\mtt m$ on $\mbb F^\infty$  {\em produces Ramsey factors} for colorings of square matrices if 
    
 \begin{enumerate}[i)]
 \item  $\mk d_{k,E_\infty}$ is a compatible metric  on $\mc D_k(E)$ uniformly equivalent to $\om_2$ on $\om_2$-bounded sets.
\item Given $k,m\in \mathbb{N}$, real numbers $\varepsilon>0$,    $  \la \ge 1$,   and a   compact metric space $(K,d_{K})$,   there is $n\in \N$ such that for every coloring  $c:(M_{n}^{k}(\mtt m;\la),d_{E_\infty^*,E_\infty})\rightarrow (K,d_{K})$ there are $R_0,R_1\in \mathcal{E}_{n, m}(\mtt m)$ such that   the restriction $\nu_{k,E_\infty}^{2}: M_{\infty}^k(\mtt m;\sma\la)\to \mc D_k(E_\infty;\sma\la)$ is an $\vep$-factor of $c$  in  $ R_0\cdot M_{m}^k(\mtt m,\sma\la) \cdot R_1^*$.

 \end{enumerate}   
\end{definition}  
 \begin{theorem}[Factorization of colorings of square matrices over $\R$, $\C$]\label{factor_p_square_matrices}
For $1\le p\le \infty$, $p\notin 2\N+4$, the $p$- norm $\nrm_p\in \mc N_\infty$    {\em produces Ramsey factors} for colorings of square matrices.
 
 \end{theorem}
 
 \subsubsection{Uniqueness}
We see now how when the metric on matrices/Grassmannians is fixed there are not so many options of being a Ramsey factor.   Suppose that $\mtt m\in \mc N_\infty$, $k\in \N$ and  $\la>1$.  A $(k,\mtt m, \la)$-Ramsey factor is a  pair $(\mu,\mc A)$ where $\mu: (M_{\infty,k}^k(\mtt m;\sma\la),d_\mtt m)\to (K_\mu,d_\mu)$ is a coloring (i.e., 1-Lipschitz  mapping) to a compact metric space $(K_\mu,d_\mu)$, $\mc A\con\mc E(\mtt m)$, and 
\begin{enumerate}[i)]
\item  The image of $\mu$ is dense in $K_\mu$.
\item $\mu(R A)=\mu(A)$ for every $R\in \mc A\cap M_{n,m}^m$ and $A\in M_{m,k}^k(\mtt m;\la)$. 
\item For every $m$, $\vep>0$ and every compact metric $(L,d_L)$ there is $n\in \N$ such that if $c: (M_{n,k}^k(\mtt m;\la),d_\mtt m)\to (L,d_L)$ is a coloring then there is some $R\in \mc A$  such that the restriction $\mu:R\cdot M_{m,k}^k(\mtt m; \sma \la)\to (K_\mu,d_\mu)$ is an $\vep$-factor of $c$ in $R\cdot M_{m,k}^k(\mtt m; \sma \la)$. 
  
\end{enumerate}  
Suppose that $\mtt m$ produces Ramsey factors for full rank matrices and set $E_\al:=(\mbb F^\al,\mtt m)$ for $\al\le \infty$. We have that  $\nu_{\mbb F^k, E_\infty}(M_{\infty,k}^k(\mtt m; \sma \la))= \mc N_k(E;\sma\la)$   is  $\partial_{E_k,E_\infty}$-totally bounded: This is because $\partial_{E_k,E_\infty}$ and $\om$ are, by hypothesis, uniformly equivalent to $\om$ on  $\mc N_k(E;\sma\la)$, and this set  is $\om$-totally bounded because is a $\om$-bounded set of $\mc N_k$. This implies that the completion $\widehat{\mc N_k(E;\sma\la)}$ is a compact space. Then it is obvious from the definition of producing Ramsey factors that $\nu_{\mbb F^k,E_\infty}: M_{\infty,k}^k(\mtt m; \sma \la)\to \widehat{\mc N_k(E;\sma\la)}$ is a $(k,\mtt m,\la)$-Ramsey factor, and in fact is the minimal one: 
 
\begin{proposition}
Suppose that $\mtt m$ produces Ramsey factors for full rank matrices, and suppose that  $(\mu,\mc A)$ is a   $(k,\mtt m, \la)$-Ramsey factor. 
\begin{enumerate}[a)]
\item  There is some surjective coloring $\theta: K_\mu\to     \wh{\mc N_k(E_\infty; \sma\la)}  $ such that $\nu_{\mbb F^k,E_\infty}= \theta \circ \mu$.
\item If $\mc A= \mc E (\mtt m)$, then there is a surjective isometry $\theta: K_\mu\to     \mc N_k(E_p; \la)  $ such that $\nu_{\mbb F^k, E_\infty}= \theta \circ \mu$.

\end{enumerate}  
\end{proposition}  
\begin{proof}
{\it a)}: For each $m$ we can find 1-Lipschitz mappings $\theta_m:(K,d_K)\to \wh{\mc N_k(E_\infty; \la)}$ and $R_m\in \mc E_{n_m, m}(\mtt m)$  such that 
$\partial_{\mbb F^k,\ell_p^\infty}(\theta_m( \mu(   R_m \cdot A)), \nu_{\mbb F^k,E_\infty}(R_m\cdot A))\le 1/2^m   $ for every $A\in  M_{m,k}^k(\mtt m;\sma\la)$. By the  coherence properties of $\mu$ and $\nu_{\mbb F^k,E_\infty}$, we obtain that
\begin{equation}\label{o834r834}
\text{$\partial_{\mbb Fk,\ell_p^\infty}(\theta_m( \mu( A)), \nu_{k,\ell_p^\infty}(A))\le 1/2^m   $ for every $A\in  M_{m,k}^k(\mtt m;\sma\la)$.}
\end{equation}  
Given $A\in M_{\infty,k}^k(\mtt m;\sma\la)$, set $x:=\mu(A)$, and let $n$ be such that $A\in M_{n,k}$. Then we know from \eqref{o834r834} that $(\theta_m( \mu( A)))_{m\ge n}$ is a Cauchy sequence, and let $\theta(x)\in \wh{\mc N_k(E_\infty;\sma \la)}$ be its limit.  Notice that $\theta(\mu(A))=\nu_{\mbb F^k,E_\infty}(A)$, so, since  $\nu_{\mbb F^k,E_\infty}(M_{\infty,k}^k(\mtt m;\sma\la))$ is dense in $\wh{\mc N_k(E_\infty; \sma\la)}$, we can conclude that $\theta$ is onto. {\it b)} is an easy consequence of {\it a)}.
\end{proof}  
With the obvious definitions of Ramsey factors for Grassmannians and for square matrices, the  corresponding statements on $\tau_{k,E_\infty}$ and $\nu_{k,E_\infty}^2$ are also true. 
\nota
For some norms $\mtt m$, for example the $p$-norms, the completion  $\widehat{\mc N_k(E;\sma\la)}$ is exactly $ \mc N_k(E;\la)$. A sufficient condition  is that for every finite dimensional subspaces  $X$ and $Y$ of $E_\infty$  there is a finite dimensional subspace $Z$ of $E_\infty$ that has  isometric copies  $X_0$ and $Y_0$ of $X$ and $Y$, respectively, such that $X_0\cap Y_0=\{0\}$. 
\fnota

\section{The proofs: Approximate Ramsey properties and extreme amenability}
In Ramsey theory, 
the usual strategy to prove that a list of colorings is the  canonical one is, given a coloring of a class of embeddings,   use the   Ramsey property for an appropriate good class of  embeddings and an enlarged number of colors that take now into account the transformation necessary to make an arbitrary embedding a good one. This is exactly what we did for  full rank matrices over a finite field. On the approximate case, one may follow the same direct approach and obtain similar results to the  ones we presented, but now  obliged to deal with several approximation arguments that make the proofs somehow unnecessarily complicated. Instead, our approach is  to apply a topological principe that is equivalent to a strong version of an {\em approximate Ramsey property}, and that makes the computations much more clear.   This is the {\em extreme amenability} of the group of linear isometries of  appropriate Banach spaces that locally are like  $\ell_p^\infty$, for $p\notin 2\N+4$.  We introduce some relevant terminology and concepts.  Recall that a Banach space is a complete normed space. Given  Banach spaces $X=(X,\nrm_X)$ and $Y=(Y,\nrm_Y)$,  and given $\de\ge 1$, let $\Emb_\de(X,Y)$ be the collection of all linear functions $T:X\to Y$ such that $(1+\de)^{-1} \nrm[x]_X\le \nrm[T (x)]_{Y}\le (1+\de)\nrm[x]_X$. Notice that when $\dim X=k<\infty$ this definition corresponds to $\mc L^{k}_{1+\de}(X,Y)$ presented before.  The following  concept was introduced in \cite[Definition 5.1]{ferenczi_amalgamation_2019} (see also \cite{bartosova_ramsey_2017} ,\cite{bartosova_ramsey_2_2019}).  
\begin{definition} 
 A family $\mc G$ of  finite dimensional normed spaces  has the \emph{Steady Approximate Ramsey $\mathit{Property}^+$} \sarpp   when for every $k\in \N$  and every $\vep>0$ there is $\de:=\de(k,\vep)> 0$ such that  if $X,Y\in \mc G$  and $\dim X=k$, then  there exists $Z\in \mc G$ such that  every  continuous coloring $c$ of $\mr{Emb}_{\de}(X,Z)$    $\vep$-stabilizes on  $\ga\circ \mr{Emb}_{\de}(X,Y)$ for some $\ga\in \mr{Emb}(Y,Z)$. 
\end{definition}  
The  \sarpp of the classes $\{\ell_p^n\}_n$ can be seen as a  multidimensional Borsuk-Ulam principle (see \cite[\S\S\S 5.1.1]{ferenczi_amalgamation_2019}). In general, for a family $\mc F$  is a strong form of amalgamation: Recall that $\mc G$ is an {\em amalgamation class} when $\{0\}\in \mc G$ and for every $\vep>0$ and $k\in \N$ there is $\de>0$ such that if $X\in \mc G$ has dimension $k$, $Y,Z\in \mc G$, and $\ga\in \Emb_{\de}(X,Y)$, $\eta\in \Emb_\de(X,Z)$, then there are $V\in \mc G$, $i\in \Emb(Y,V)$ and $j\in \Emb(Z,V)$ such that $\nrm[i\circ \ga -j \circ \eta]\le \vep$. It is not difficult to see that if $\mc F$ has the \sarpp then it is an amalgamation class.

 To an amalgamation class $\mc G$ it corresponds a unique separable ``generic'' Banach space  $E$ whose family of finite dimensional substructures, denoted by $\age(E)$, is minimal containing $\mc G$. This is the content of the Fraïssé correspondence on the category of Banach spaces.  We write $\mc G_E$ to denote the class of  subspaces of $E$ that are isometric to some element of $\mc G$, and $\overline{\mc G}^\con$ to denote the class of subspaces of elements of $\mc G$; we write  $X\in \mc G_\equiv$ when some element of $\mc G$ is isometric to $X$, and we say that $\mc G$ is hereditary if $Y\in \mc G$, and $\Emb(X,Y)\neq \buit$, then $X\in \mc G_\equiv$.  Finally,  $\mc G\preceq \mc H$ means that every space in $\mc G$ is isometric to some element of $\mc H$, and $\mc G\equiv \mc H$  to denote that $\mc G\preceq \mc H\preceq \mc G$. Note that if  $\mc G\equiv \mc H$, then $\mc G$ has the \sarpp (is an amalgamation class) if and only if  $\mc H$ has the \sarpp (resp. is an amalgamation class).  
\begin{theorem}[Fraïssé correspondence;     \cite{bartosova_ramsey_2_2019},  \cite{ferenczi_amalgamation_2019}]
Let $\mc G$ be a  class of finite dimensional normed spaces. 
\begin{enumerate}[a)]
\item If $\mc G$ is an amalgamation class, then there is a unique separable Banach space $E$, called the {\em $\mc G$-Fraïssé limit}, and    denoted by $\flim\mc G$, such that  $\overline{\mc G_E}^\con$  is $\La_E$-dense in $\age(E)$ and $E$ is {\em Fraïssé}, that is for every $\vep>0$ and $k\in \N$ there is $\de>0$ such that  the natural action $\iso(E)\acts \Emb_{\de}(X,E)$ is $\vep$-transitive (given $\ga,\eta\in \Emb_\de(X,E)$ there is $g\in \iso(E)$ such that $\nrm[g\circ \eta- \ga]\le \vep$).   
 \item The following are equivalent:
 \begin{enumerate}[i)]
 \item  $\mc G$ is hereditary amalgamation class that is  $d_\mr{BM}$-compact, that is, for every $k$,  the collection of classes $[\mtt m]$ of norms  $\mtt m \in \mc N_k$ such that $(\mbb F^k, \mtt m)\in \mc G_\equiv$ is a closed subset of $\mc B_k$. 
 \item There is a unique separable Fraïssé Banach space $E$ such that $\age(E)\equiv \mc G$.
   
 \end{enumerate}   
\end{enumerate}  
\end{theorem}     
This can be considered as the Banach space version of the Fraïssé correspondence of first order structures, that, for example, interprets several Random graphs (Rado, Henson graphs), Boolean algebras (the countable atomless one), or metric spaces (the rational Urysohn space)  as Fraïssé limits.  The  known examples of families having the \sarpp  are related to the $p$-norms:  
\begin{enumerate}[$\bullet$]
\item    $\{\ell_p^n\}_n$ for all $1\le p\le \infty$: For $p=2$, this is a consequence of the fact, via the Kechris-Pestov-Todorcevic (KPT) correspondence (see \cite[Theorem 2.12]{bartosova_ramsey_2_2019}, \cite[Theorem 5.10]{ferenczi_amalgamation_2019}), that the unitary group $\iso(\ell_2)$ is extremely amenable, proved by M. Gromov and V. Milman \cite{gromov_topological_1983}, and the fact that $\{\ell_2^n\}_n$  is an amalgamation class   (see for example \cite[Example 2.4.]{ferenczi_amalgamation_2019}).  The case  $1\le p\neq 2<\infty$ follows from the approximate Ramsey property of this class, proved in \cite{ferenczi_amalgamation_2019}      and the result of G. Schechtman in \cite{Schechtman_1979}  stating that $\{\ell_p^n\}_n$  are amalgamation classes. The case $p=\infty$ is proved in \cite{bartosova_ramsey_2_2019} (see also \cite{bartosova_2018}) using  the dual Ramsey Theorem.  
 
\item $\age(L_p[0,1])$ for $p\notin 2\N+4$:  This    is a byproduct of  the extreme amenability of $\iso(L_p[0,1])$, proved by T. Giordano and V. Pestov in \cite{giordano_extremely_2007}, the (KPT) correspondence,  and the fact that    $\age(L_p[0,1])$  is an amalgamation class, proved in \cite{ferenczi_amalgamation_2019}.  On the other direction, when $p\in 2\N+4$, it is  shown in \cite[Proposition 2.10.]{ferenczi_amalgamation_2019}  that $\age(L_p[0,1])$  does not have the \sarpp because in these spaces there are finite dimensional isometric subspaces, one well complemented and the other badly complemented. 
\item   $\mc F=\age(C[0,1])$: This is proved in   \cite{bartosova_ramsey_2_2019} (see also \cite{bartosova_2018}), directly using injective envelopes and some approximations, or as a byproduct of the \sarpp of the family $\{\ell_\infty^n\}_n$ and the Kechris-Pestov-Todorcevic correspondence for Banach spaces.  
\end{enumerate}  

The \sarpp characterizes norms on $\mbb F^\infty$ that produce Ramsey factors.  

\begin{theorem} \label{lk3io2ior4o23i32}
Let  $\mtt m$ be a norm on $\mbb F^\infty$,  $E:=(\mbb F^\infty,\mtt m)$.  
\begin{enumerate}[a)]
\item  If $\age(E)$ has the \sarpp, then  $\mtt m$ produces Ramsey factors for colorings of full-rank matrices, Grassmannians and square matrices. 
\item  If $\mtt m$ produces Ramsey factors for colorings of full rank matrices,  $\age(E)$ has the  \sarpp.
\end{enumerate}
 
\end{theorem} 

To prove {\it b)} we will use the following.
\lema \label{io898989866}   Let  $\mtt m$ be a norm on $\mbb F^\infty$ that produces Ramsey factors for colorings of full rank matrices, set $E:=(\mbb F^\infty,\mtt m)$. For every     $k,m,r\in \mathbb{N}$,   $\varepsilon>0$,  and  $\la\in ]1,\infty[$   there is some $n\in\N$ such that for every discrete coloring $c: M_{n,k}^k(\mtt m;\la)\to r$ there is some $R\in \mc E_{n,m}(\mtt m)$ such that 
\begin{equation}\label{lioj4rioj4wiorjower}
 \text{$R \cdot B \in (c^{-1}(c(R\cdot A)))_{\partial_{(\mbb F^k,\mtt m), E}(\nu_{\mbb F^k, E}(A),\nu_{\mbb F^k,E}(B))+\vep}$
 for every $A,B\in M_{m,k}^k(\mtt m;\sma\la)$} 
\end{equation} 
  \flema	 
 \begin{proof} 
Fix the parameters $k,m,r\in \N$, $\vep$,  and $\la$. Let $n\in \N$  be the outcome  of property   {\it ii)} in Definition \ref{io4jrio4rfer3w}   when applied to   $k,m$, $\vep/2$, $\la$ and the compact metric space $K:= 2\la_1\Ball(\ell_\infty^r)$. We claim that $n$ works. Fix $c:M_{n,k}^k(\mtt m;\la)\to r$, and let $f: M_{n,k}^k(\mtt m;\la)\to K$,  $f(A):= (d(A, c^{-1}(j)))_{j<r}$. It is clear that $f$ is 1-Lipschitz, so there is some $R\in \mc E_{n,m}(\mtt m)$  and $\widetilde{f}: \mc N_k(E;\sma\la)\to K$ such that $d_K(\wt{f}(\nu_{\mbb F^k,E}(A)), f(R\cdot A))\le \vep/2$ for every $A\in M_{m,k}^k(\mtt m;\sma\la)$.   Fix $A,B\in M_{m,k}^k(\mtt m;\sma\la)$. Then, $d_K(f(R\cdot A), f(R\cdot B))\le \partial_{(\mbb F^k,\mtt m), E}(\nu_{\mbb F^k,E}(A), \nu_{\mbb F^k,E}(B))+\vep$. Thus, if $j:=c(R\cdot A)$,   then the $j^\mr{th}$-coordinate of $f(R\cdot A)$ is zero, hence, the $j^\mr{th}$-coordinate of $c(R\cdot B)$  must satisfy that $d(R \cdot B, c^{-1}(j))\le   \partial_{(\mbb F^k,\mtt m),E}(\nu_{\mbb F^k,E}(A), \nu_{\mbb F^k,E}(B))+\vep$, as desired. 
 \end{proof} 
 
\begin{proof}[Proof of {\it b)} of Theorem \ref{lk3io2ior4o23i32}] 
The proof of {\it a)} is more involved, and it will be done in several steps later.   
Let $\mc F$  be the collection of  normed spaces of the form $(\mbb F^k, \mtt n)$  with $\mtt n\in \mc N_k(E)$ and such that $\om(\mtt n, \nrm_1)\le k$. Since the diameter of the Banach-Mazur compactum $\mc B_k$ is at most $k$, it follows that $\mc F\equiv \age(E)$, so the \sarpp of $\mc F$ and of $\age(E)$ are equivalent.  Moreover, we prove the following equivalent discrete version of the \sarpp (see \cite{bartosova_ramsey_2_2019}, \cite{ferenczi_amalgamation_2019}): 

For every $k$ and $\vep>0$  there is a $\de>0$ such that for every $r\in \N$ and  $X,Y\in \mc F$ with $\dim X=k$  there is $Z\in \mc F$ such that   every discrete coloring $c:\Emb_\de(X,
Z)\to r$  has an $\vep$-monochromatic set of the form $R\circ \Emb_\de(X,Y)$ for some  $R\in \Emb(Y,Z)$.

 Fix a dimension $k$ and $\vep>0$. Notice that the collection  of spaces in $\mc F$ of dimension $k$ is a $\om$-bounded set, so, by hypothesis, the metrics $\partial_{(\mbb F^k, \mtt m), E}$ and $\om$ are uniformly equivalent on $\mc M:=\conj{\mtt n\in \mc N_k(E)}{  (\mbb F^k,\mtt n)\in \mc F}$.
  Let $\de>0$ be such that if $\mtt n,\mtt p\in \mc M$ are such that $\om(\mtt n,\mtt p)<\de $, then $\partial_{(\mbb F^k, \mtt m), E}(\mtt n,\mtt p)<\vep/2$.  We claim that $\de$ works. For suppose that $X=(\mbb F^k,\mtt n),Y=(\mbb F^m,\mtt p)\in \mc F$ are such that $\Emb(X,Y)\neq \buit$, and $r\in \N$. Let $m_0\ge m$ and $C\in M_{m_0,m}^m$  be such that $\mtt n=\nu_{\mbb F^m,(\mbb F^{m_0}, \mtt m)}(C)$, and let  $\la> 1$ be such that  $T_C\circ\Emb_{\de}(X,Y)\con \conj{T_B}{  B\in M_{m_0,k}^k(\mtt m; \sma\la)}$. We use  Lemma \ref{io898989866}   for the parameters   $k,m_0,r+1$, $\vep/2$ and $\la$ to find a corresponding $n$;  set $Z:=(\mbb F^n,\mtt m)$. Now  suppose that $c:\Emb_{\de}(X,Z)\to r$. We define $\widehat{c}:M_{n,k}^k(\mtt m;\la)\to r+1$ by $\widehat{c}(A)=c(T_A)$ if $T_A\in \Emb_\de(X,Z)$ and  by  $\widehat{c}(A)=r$ otherwise. Let $R\in \mc E_{n,m_0}(\mtt m )$ be such that \eqref{lioj4rioj4wiorjower} holds. Let $\ga:=T_R\circ T_{C}\in \Emb(Y,Z)$.  We   see that $\ga\circ \Emb_\de(X,Y)$ is $\vep$-monochromatic for $c$: Fix  $\eta\in \Emb(X,Y)$, and let  $A\in M_{m_0,k}^k(\mtt m; \sma\la)$ be such that $T_A= T_C\circ \eta$, and let $j:= c(T_R\circ T_C\circ \ga)=\widehat{c}(R\cdot A)$. Now given  $\xi\in \Emb_\de(X,Y)$,  let  $B\in M_{m,k}^k(\mtt m;\sma\la) $ be such that $T_B=T_C\circ \xi$.  Then,     $\mtt n:= \nu_{\mbb F^k,E}(A)$,  $\mtt p:= \nu_{\mbb F^k,E}(B)$ and therefore $\om( \nu_{\mbb F^k,E}(A), \nu_{\mbb F^k,E}(B)) \le \de$, hence, $\partial_{(\mbb F^k,\mtt m),E}(\nu_{\mbb F^k,E}(A),\nu_{\mbb F^k,E}(B))\le \vep/2$. This together with  \eqref{lioj4rioj4wiorjower} gives that $R\cdot B\in (\widehat{c}^{-1}(j))_\vep$, so there must be $D\in M_{n,k}^k(\mtt m;\la)$ such that $\widehat{c}(D)=j$ and such that $\nrm[T_D- \ga\circ \xi]=\nrm[T_D- T_R\circ T_B]=  d_{\mtt m}(D,R\cdot B)\le \vep$; since $j<r$, $T_D\in \Emb(X,Z)$, so $c(T_D)=\widehat{c}(E)=j$ and  we are done. 
\end{proof}  
The proof of {\it b)} of Theorem \ref{lk3io2ior4o23i32} has two main parts.   The   first one (Theorem \ref{43iooi4355}) uses the fact that if $\mc F$ has the \sarpp and it is hereditary, then the isometry group $G$ of  the Fraïssé limit $\flim\mc F$ is extremely amenable with its topology of pointwise convergence. 
This property will be used as   infinitary   principles can be used to conclude, via compactness arguments,  the finitary ones (e.g. infinite vs finite Ramsey, Hindman vs Folkman theorem, etc.). The fix point property of $G$ will  naturally provide   abstract Ramsey factors that are  $G$-quotients.  The second part of the argument   is to see that  these $G$-quotients  are in fact the desired Ramsey factors.

Recall that a topological group $G$ is called \emph{extremely amenable} when 
every continuous action of $G$ on a compact Hausdorff space has a fixed
point.   There is a useful characterization of extreme amenability in terms of factors through quotients that we pass to explain. 

Let $(M,d)$ be a metric space, and let    $G\curvearrowright M$ be a continuous action by isometries.  We write $[p]_{G}$ to denote
the closure of the $G$-orbit of $p\in M$, and $M\quo G$ to
denote the space of closures of $G$-orbits of $M$. Since $G$ acts by
isometries the formula 
\begin{equation*}
\widetilde{d}^G([p],[q]):=\inf \{{d_{M}(p_0, q_0)}\,:\,{p_0\in
[p],\, q_0\in [q]}\}
\end{equation*}%
defines the quotient pseudometric induced by the quotient map $\pi:M\mapsto M/\hspace{-0.1cm}/G$, and as  we consider \emph{closures}
of orbits, $\widetilde{d}^{G}$ is a metric. It is easily seen that $\widetilde{d}^{G}$ is complete if $d$ is complete.

Given a compact metric space $(K,d_{K})$,  let $\mathrm{Lip}%
((M,d_{M}),(K,d_{K}))$ be the collection of all $K$-colorings of $M$. With
the topology of pointwise convergence $\mathrm{Lip}((M,d_{M}),(K,d_{K}))$ is
a compact space, which is metrizable when $(M,d_{M})$ is separable. The
continuous action $G\curvearrowright (M,d_{M})$ induces a natural continuous
action $G\curvearrowright \mathrm{Lip}((M,d_{M}),(K,d_{K}))$, defined by
setting $(g\cdot c)(p):=c(g^{-1}\cdot p)$ for every $c\in \mathrm{Lip}%
((M,d_{M}),(K,d_{K}))$ and $p\in M$. This is the aforementioned characterization (see \cite{bartosova_ramsey_2_2019}).

\begin{proposition}
\label{factor_orbitspace0} Suppose that $G$ is a Polish group, and $d_{G}$
is a left-invariant compatible metric on $G$. The following assertions are
equivalent.

\begin{enumerate}[i)]
\item $G$ is extremely amenable.

\item The left translation of $G$ on $\left( G,d_{G}\right) $ is finitely
oscillation stable \cite[Definition 1.1.11]{pestov_dynamics_2006}, that is, for every continuous coloring  $c: M\to [0,1]$ and every $F\con M$ finite and $\vep>0$ there is some $g\in G$ such that $\osc(c\rest g \cdot F)\le \vep$.  

\item For every action $G\curvearrowright M$ of $G$ on a metric space $%
(M,d_{M})$, and for every compact coloring $c:(M,d_{M})\rightarrow (K,d_{K})$
of $(M,d_{M})$, there exists a compact coloring $\wh{c}:M/\hspace{-0.1cm%
}/G\rightarrow K$ such that for every finite $F\subseteq M$ and  $%
\varepsilon >0$ there is some $g\in G$ such that $d_{K}(c(p),\wh{c}%
([p]_{G}))<\varepsilon $ for every $p\in g\cdot F$.

\item The same as iii) where $F$ is compact.
\end{enumerate}
\end{proposition}
 \begin{proof}
The equivalence of \textit{i)} and \textit{ii)} can be found in \cite[%
Theorem 2.1.11]{pestov_dynamics_2006}. The implication \textit{iii)}$%
\Rightarrow $\textit{ii)} is immediate, since orbit space $G/\hspace{-0.1cm}%
/G$ is one point. We now establish the implication \textit{i)}$\Rightarrow $%
\textit{iv)}: Fix a 1-Lipschitz $c:(M,d_{M})\rightarrow (K,d_{K})$. Let $L$
be the closure of the $G$-orbit of $c$ in $\mathrm{Lip}((M,d_{M}),(K,d_{K}))$%
. By the extreme amenability of $G$, there is some $c_{\infty }\in L$ such
that $G\cdot c_{\infty }=\{c_{\infty }\}$, so we can define the  quotient $K$-coloring $\wh{c}([p]_{G}):=c_{\infty }(p)$. Given a compact subset $F$ of 
$M$, let $g\in G$ be such that $\max_{p\in F}d_{K}(c_{\infty }(p),c(g\cdot
p))<\varepsilon$. If $x\in F$, then $d_{K}(c(g\cdot x),\wh{c}([g\cdot
x]_{G}))=d_{K}(c(g\cdot x),c_{\infty }(x))<\varepsilon$.
\end{proof}

We apply this characterization to groups of linear isometries of a Banach space. 
Given two Banach spaces $X$ and $Y$, recall that   $\mc L(X,Y)$ is  the Banach space of all bounded linear operators $T:X\to Y$, endowed with the operator norm $\nrm[T]:=\sup_{\nrm[x]_X=1} \nrm[Tx]_Y$, and when  $\mathrm{Im}(T)$ of $T$ is  finite-dimensional,  $\nrm[T^{-1}]:=\min\conj{a\ge 0}{ \mr{Ball}({TX})\con a T(\mr{Ball}(X)) }$.
 In this case, $\nrm[T^*]=\nrm[T]$ and $\nrm[(T^*)^{-1}]=\nrm[T^{-1}]$.
The special case when $T:X\to Y$ is 1-1 and $\nrm[T]=\nrm[T^{-1}]=1$ corresponds to $T$ being an isometric embedding. The collection of such maps is denoted by $\mr{Emb}(X,Y)$. 
Let $\mc L_{\la}(X,Y)$, $\mc L_{\sma \la}(X,Y)$,  be  the set of all $T\in \mc L(X,Y)$ with finite dimensional image   such that $\nrm[T],\nrm[T^{-1}]\le \la$, resp. $<\la$. Let  
 $\mc L^k(X,Y)$ be the set of all  $T\in \mc L(X,Y)$  whose image is $k$-dimensional, and let   $\mc L_{\lambda }^k(X,Y)=\mc L_{\lambda }(X,Y)\cap \mc L^k(X,Y)$, $\mc L_{\sma\lambda }^k(X,Y)=\mc L_{\sma\lambda }(X,Y)\cap \mc L^k(X,Y)$. Let $\mc L^{k, \mr{w}^*}(X^*,X)$ be the metric space  of operators $T\in \mc L^k(X^*,X)$ such that $T$ admits a full rank  decomposition, i.e., when $T=T_0\circ T^*_1$ for some $T_0, T_1\in \mc L^k(\mbb F^k, X)$. It is an exercise to prove that this is equivalent to saying that  $T$ is a $\mr{w}^*$-to-norm continuous linear   operators from $X^*$ to $X$  of rank $k$; let   $\mc L^{k, \mr{w}^*}_\la(X^*,X):= \mc L^{k, \mr{w}^*}(X^*,X)\cap  \mc L_\la(X^*,X)$.
 \begin{definition}
Let $\iso(E)\acts \mc L(X,E)$ be the canonical action by isometries $g\cdot T:= g\circ T$,
$\iso(E) ^2 \acts \mc L^{k,\mr{w}^*}(E^*,E)$ be the canonical action by isometries $(g,h)\cdot T:= g \circ T \circ h^*$ for   $(g,h)\in \iso(E)^2$ and $T\in \mc L^{k,\mr{w}^*}(E^*,E)$, and let   $\iso(E)\acts \mr{Gr}(k,E)$ be the canonical action by  isometries $g \cdot V:= g(V)$.
   
 \end{definition}  
Note  that  $\mc L_{\lambda }^k(X,E)$,   $\mc L_{\sma \lambda }^k(X,E)$, and  $\mc L^{k, \mr{w}^*}_\la(X^*,X)$, $\mc L^{k, \mr{w}^*}_{\sma\la}(X^*,X)$ are $\iso(E)$-closed and   $\iso(E)^2$-closed, respectively. The next readily follows from Proposition \ref{factor_orbitspace0}, using the fact that if $G$ is extremely amenable, then $G^2$ is also extremely amenable (see \cite[Corollary 6.2.10.]{pestov_dynamics_2006}).

\begin{lemma}
\label{oi34ioioirffvfgg} Suppose that $X,E$ are Banach spaces, $X$ is finite-dimensional, and  $\iso(E)$ is extremely amenable. Let $k\in \N $,  $\vep>0$, $1\le\la$,      $Y\in \age(E)_\equiv$, and   let $(K,d_K)$ be  a compact metric space.   
\begin{enumerate}[a)]
\item For every $K$-coloring  $c$ of   $(\mathcal{L}_{\sma\la}^k(X,E),d_{X,E})$   there is  $R\in \mathrm{Emb}(Y,E)$
 such that   the quotient map $\pi: \mc L_{\sma\la}^{k}(X,E)\to \mc L_{\sma\la}^k(X,E)\quo \iso(E)$ is an $\vep$-factor of  $c$ in $R \circ  \mc{L}_{\sma\la}^k(X, Y)$.
\item For every $K$-coloring  $c$ of $(\gr(k,E), \La_E)$ there exists $V\in \gr(\dim Y, E) $
 with $(V,\nrm_E)$ is isometric to $Y$  such that  the quotient map $\pi: \gr(k,E)\to \gr(k,E)\quo \iso(E)$  is an $\vep$-factor of  $c$ in $\gr(k,V)$.
 \item For every $K$-coloring  $c$ of $(\mc L^{k,\mr{w}^*}_{\sma\la}(E^*,E), d_{E^*,E})$ there   are  $R_0,R_1\in \Emb(Y,E)$ such that
    the quotient map $\pi:\mc L^{k,\mr{w}^*}_{\sma\la}(E^*,E)\to\mc L^{k,\mr{w}^*}_{\sma\la}(E^*,E)\quo \iso(E)^2$ is an $\vep$-factor of  $c$ in $R_0 \circ \mc L^{k,\mr{w}^*}_{\sma\la}(Y^*,Y)\circ R_1^*$.
\qed
 \end{enumerate}

\end{lemma}

The relationship between the \sarpp  of a clas  of finite dimensional normed spaces and the extreme amenability of the isometry group of its Fraïssé limit  is the next  mix of the Fraïssé and the Kechris-Pestov-Todorcevic correspondences, that we took from \cite[Corollary 5.11]{ferenczi_amalgamation_2019}.

\teor \label{43iooi4355} 
If   $\mc G$ is an hereditary family with the \sarpp, then  the Banach-Mazur closure  of $\mc G$ also has the \sarpp and   the Fraïssé limit $\flim \mc G$ is a Fraïssé Banach space whose isometry group is  extremely amenable  with its strong operator topology. 
  \qed
  \fteor
\begin{definition}
 Given a normed space $E:=(\mbb F^\infty, \mtt m)$ such that $\age(E)$ is an amalgamation class, we write  $\widehat{E}$ to denote, the Fraïssé limit $\flim \age(E)$.  
\end{definition}

 We will use the following notation. Given a normed space $E=(\mbb F^\infty,\mtt m)$ and $n\in \N$, we set $E_n:=(\langle u_j\rangle_{j<n},\mtt m)$, and given a normed space $X$ we write $\age(X)_\mtt m$ to denote the collection of subspaces of $X$ isometric to some $E_n$. The following is the asymptotic version of Lemma \ref{oi34ioioirffvfgg}.

\begin{corollary}
\label{finitiz1}
Suppose that  $E=(\mbb F^\infty, \mtt m)$ is such that $\age(E)$ has the \sarpp,  and  $\age(\wh E)_\mtt m$ is an amalgamation class.   Let   $k,m\in \mathbb{N}$,   $\vep>0$ and $1<\la$.      Given a  compact metric space $(K,d_{K})$   there is $X\in \age(\wh E)_\mtt m$   such that  
  
\begin{enumerate}[1)]
\item for every $K$-coloring  $c$ of   $\mathcal{L}_{\la}^k(E_k,X)$   there is   $R\in \mathrm{Emb}(E_m,X)$
 such that   the quotient map $\pi: \mc L_{\sma\la}^k(E_k,\wh E)\to \mc L_{\sma\la}^k(E_k,\wh E)\quo \iso(\wh E)$ is an $\vep$-factor of  $c$ in $R \circ  \mc{L}_{\sma\la}^k(E_k,  E_m)$;

\item     for every $K$-coloring  $c$  of $(\gr(k,X), \La_{X})$ there  is $V\in \gr(m, X)\cap \age(\wh E)_\mtt m$ such that  the quotient map $\pi: \gr(k,\wh E)\to \gr(k,\wh E)\quo \iso(\wh E)$  is an $\vep$-factor of  $c$ in $\gr(k,V)$;
 \item   for every $K$-coloring  $c$ of $\mc L^{k,\mr{w}^*}_{\la}(X^*,X)$ there   are  $R_0,R_1\in \Emb(E_m,X)$ such that
    the quotient map $\pi:\mc L^{k,\mr{w}^*}_{\sma \la}(\wh E^*,\wh E)\to\mc L^{k,\mr{w}^*}_{ \sma\la}(\wh E^*,\wh E)\quo \iso(\wh E)^2$ is an $\vep$-factor of  $c$ in $R_0 \circ \mc L^{k,\mr{w}^*}_{ \sma\la}(E_m^*,E_m)\circ R_1^*$.

 \end{enumerate}

 \end{corollary}

\begin{proof}   
For each $X\in \age(\wh E)$, and each $\vep>0$, let $A_{X,\vep}$ be the collection of all $Y\in \age(\wh E)_\mtt m$ such that $X\con_\vep Y$.  Then $\{A_{X,\vep}\}_{A,\vep}$ is a family of subsets of $\age(\wh E)_\mtt m$ with the finite intersection property. Let $\mc U$ be a non principal ultrafilter on $\age(\wh E)_\mtt m$ containing all $A_{X,\vep}$. 
Now suppose   for the sake of contradiction, that, for some compact
space $(K,d_{K})$, there is no such $X\in \age(\wh E)_\mtt m$ satisfying {\it1)}, {\it2)} or {\it3)}. Since $\mc U$ is an ultrafilter there is $j=1,2,3$ such that the set $B_j:=\conj{X\in \age(\wh E)_\mtt m}{ X\text{ does not satisfy {\it j)}}}$ belongs to $\mc U$. Suppose that $j=1$. For each $X\in B_1$ there exists a  coloring $c_{X}:\mathcal{L}_{\la}^k(E_k,X)\rightarrow K$ providing a counterexample.  For each $T\in \mathcal{L}_{\sma\la}^k(E_k ,\wh E)$, let $c(T)\in K$ be defined as follows.   We say that $c(T)=x\in K$ if and only if for every $\vep>0$  one has that  
$\conj{Y\in C_{T,\vep}}{ \text{ such that  $T\in (\mc L_{\la}^k(E_k, Y))_\vep$ and    $d_K((T)_\vep\cap \mc L_{\la}^k(E_k, Y), x)\le \vep$}}\in\mc U.$
This is well defined   because  $K$ is compact and     the set of $Y\in  B_1$ such that   $T\in (\mc L_{\la}^k(E_k, Y))_\vep$  belongs to $\mc U$.  It is easy to see that $c$ defines a coloring, i.e., that $c$ is 1-Lipschitz. Let $\pi:\mc L_{\sma\la}^k(E_k,\wh E)\to\mc L_{\sma\la}^k(E_k,\wh E)\quo \iso(\wh E)$ be the quotient mapping.  By Lemma \ref{oi34ioioirffvfgg} there exist $S\in \mathrm{Emb}(E_m,\wh E)$ and a  coloring $\widehat{c}:(\mc L_{\sma\la}^k(E_k,\wh E)\quo \iso(\wh E), \widehat{d})\rightarrow (K,d_{K})$ such that  $d_{K}(c(S\circ T),\widehat{c}(\pi(T))\leq  \vep/2 \text{ for every $T\in \mathcal{L}_{\sma\la}^k(E_k,E_m)$}$.
Since $\age(\wh E)_\mtt m$ is an amalgamation class, the set $C$ of $Y\in \age(\wh E)_\mtt m$ such that 
there is $S_Y\in \Emb(E_m, Y)$ such that $\nrm[S_Y- S]\le \vep/3$ belongs to $\mc U$,  and since  
  $\mc L_{\sma\la}^k(E_k,E_m)$ is pre-compact, the set $D$ of those $Y\in C$ such that  
  $\max\conj{d_K(c_Y(S_Y\circ T), c(S\circ T))}{T\in \mc L_{\sma\la}^k(E_k,E_m)}\le \vep/2$ also belongs to $\mc U$. Then $Y\in D$, $S_Y$ and $\widehat{c}$ contradicts the assumption that $c_Y$ is a counterexample. Fix $T\in \mc L_{\sma\la}^k(E_k,E_m)$.  It follows that 
\begin{align*} \pushQED{\qed}
d_{K}(\widehat{c}(\pi(  T)),c_{Y}(S_{Y}\circ T)) = & d_{K}(\widehat{c}(\pi(S_{Y}\circ T)),c_{Y}(S_{Y}\circ T))\leq  d_{K}(\widehat{c}(\pi(S\circ T)),c(S\circ
T))+ \\
&+  d_{K}(c_{Y}(S_{Y}\circ T),c(S\circ T))\leq  \vep
\end{align*}
The cases $j=2,3$ are proved similarly, so we leave the details to the reader. 
 \fprue

\subsection{Orbit spaces for Fraïssé Banach spaces}\label{892sdfwe}
We see that the orbit spaces considered in Corollary \ref{finitiz1}  {\it1}, {\it2}, and {\it3}, are homeomorphic to $\mc N_k(\wh{E})$, $\mc B_k(\wh E)$ and $\mc D_k(\wh E)$, respectively.
We also show that the $\wh E$-extrinsic  metrics extend the corresponding $E$-extrinsic ones, finishing  the proof of Theorem \ref{lk3io2ior4o23i32} {\it a)}.

\teor\label{Canonical_orbit_metrics}
Suppose   $E=(\mbb F^\infty,\nrm_E)$ is such that $\age(E)$ is an amalgamation class. Then,  
 \begin{enumerate}[a)]
\item $\partial_{X,\wh E}$ is a compatible metric on $\mc N_X(\wh E)$ that is uniformly equivalent to $\om$ on $\om$-bounded sets, 
  $\partial_{X, E}=\partial_{X,\wh E}$ on $\mc N_X(E)$, and  $\mc N_X(E)$ is dense in $\mc N_X(\wh E)$ for every $X$.

 \item   $\ga_{k,\wh E}$ is a compatible metric on $\mc B_k(\wh E)$,
   $\ga_{k,\wh E}=\ga_{k,E}$ on $\mc B_k(E)$, and $\mc B_k(E)$ is dense in $\mc B_k(\wh E)$.
 \item  $\mk d_{k,\wh E}$  is a compatible metric on $\mc D_k(\wh E)$ that is uniformly equivalent to $\wt\om_2$ on  $\wt\om_2$-bounded sets, 
  $\mk d_{k, E} = \mk d_{k,\wh E} $ on  $  \mc D_k(E)$, and $\mc D_k(E)$ is dense in $\mc D_k(E)$. 
 \end{enumerate}
\fteor

\begin{proof}

 {\it a)}: Suppose that $\dim X=k$. Recall that we consider $\mc N_X$ with its natural topology of pointwise convergence.   The mapping
 $\nu_{X,\wh E}: \mc L^k(X,\wh E)\to \mc N_X(\wh E)$ is continuous, because    the  convergence in norm implies pointwise converge.  We see  that  $\nu_{X,\wh E}(T)=\nu_{X,\wh E}(U)$ if and only if $[T]=[U]$. The reverse implication is clear; now suppose that   $\nu_{X,\wh E}(T)=\nu_{X,\wh E}(U)$. Let $Y:=T(X)$ be endowed with the $\wh E$-norm, and let $\theta: Y\to X$ be the inverse isometry of $T:X\to Y$. Then $U\circ \theta\in \Emb(Y,E)$; so, given $\vep>0$, there is a global isometry $\al$ of $E$ such that $\nrm[U\circ \theta-\al\rest Y]\le \vep$, or equivalently, $\nrm[U-\al \circ T]\le \vep$. Since $\vep$ is arbitrary, we obtain that $U\in [T]$.  We   show that $\widetilde{\nu}_{X,\wh E}$ is a homeomorphism. Suppose that $(\mtt m_j)_j$ is a converging sequence in $\mc N_X(\wh E)$ with limit $\mtt m\in \mc N_X(\wh E)$.   For each $j$, let $T_j\in \mc L^k(X,\wh E)$ be such that $\nu_{X,\wh E}(T_j)=\mtt m_j$, and let $T\in \mc L^k(X,\wh E)$ be such that $\nu_{X,\wh E}(T)=\mtt m$.
\clam
$([T_j])_j$ is a Cauchy sequence.
\fclam
Notice that it follows from this, and the fact that  the quotient metric $\widetilde{d}_{X,\wh E}$ is complete (here we use that $X$ and $\wh E$ are Banach spaces), that $([T_j])_j$ converges to some $[U]$; by the continuity of $\nu_{X,\wh E}$ we have that $\nu_{X,\wh E}(U)=\mtt m=\nu_{X,\wh E}(T)$, so $([T_j])_j$ converges to $[U]=[T]$.   Also, given a bounded subset $A$ of $\mc N_X(\wh E)$, its closure  $\overline{A}$ is compact, so $\partial_{X,\wh E}$ and $\om$ are uniformly equivalent on $\overline{A}$, thus also on $A$.

Let us prove the previous claim:   Set $Y:=T(X)$, normed as subspace of $\wh E$, and let $\theta:Y\to X$ be the inverse isometry of $T:X\to Y$,  and fix $\vep>0$; since $\wh E$ is  weak-Fraïssé,   there is some $\de>0$ such that  the canonical action  $\iso(\wh E)\acts \Emb_\de(Y,\wh E)$ is $\vep$-transitive;  let $j_0$ be such that $T_j \circ \theta \in \mr{Emb}_\de(Y, \wh E)$ for every $j\ge j_0$; this means that if $j_1,j_2\ge j_0$, then there is $\al \in \iso(\wh E)$ such that 
$\nrm[T_{j_1}  -\al  \circ T_{j_2}  ]   =\nrm[T_{j_1} \circ \theta -\al  \circ T_{j_2} \circ \theta ]\le \vep,$ 
hence $\widetilde{d}_{X,\wh E}([T_{j_1}],[T_{j_2}])\le \vep$.

Let us prove now that $\partial_{X, E}(\mtt m,\mtt n)=\partial_{X,\wh E}(\mtt m,\mtt n)$:  Since $E$ is isometrically embedded into $\wh E$, we have that  $\partial_{X, \wh E}(\mtt m,\mtt n)\le \partial_{X,E}(\mtt m,\mtt n)$. Now suppose that $\nu_{X,\wh E}(T)=\mtt m$, $\nu_{X,\wh E}(U)=\mtt n$, and $\vep>0$. To simplify the notation, let $X_\mtt m:=(X,\mtt m)$ and $X_\mtt n:=(X,\mtt n)$. 
 We use the fact that  $\age(E)$ is an amalgamation class to find $\de>0$ such that for every $Y,Z,V$ that can be isometrically embedded into $E$, with $\dim Y=k$, and every $\ga\in \Emb_\de(Y,Z)$, $\eta\in \Emb_\de(Y,V)$ there is $W$ that can be isometrically embedded into $E$ and $i\in \Emb(Z,W)$, $j\in \Emb(V,W)$ such that $\nrm[i\circ \ga -j\circ \eta]\le \vep$.   Since $\age(E)_{\wh E}$ is $\La_{\wh E}$-dense in $\age(\wh E)$, we can find $Z\in  \age(E)_{\wh E}$ such that there is $\theta \in \Emb_{\de}( Y, Z  )$ such that $\nrm[\theta - i_{Y, \wh E}]\le \vep$, where $Y:=\im T +\im U$.  Then $T_0:=\theta\circ T\in \Emb_\de( X_\mtt m, Z)$, $U_0:=\theta\circ U\in \Emb_\de(X_\mtt n, Z)$  and setting $K_\mtt m:= \nrm[\id_X]_{X, X_\mtt m}$, $K_\mtt n:= \nrm[\id_X]_{X, X_\mtt n}$, 
\begin{equation*}\label{kjehjehe}
\nrm[T_0-U_0]_{X,Z}\le   \nrm[T-T_0]_{X_\mtt m, \wh E} K_\mtt m+ \nrm[T-U]_{X,\wh E} +\nrm[U-U_0]_{X_\mtt n, \wh E} K_\mtt n \le \vep (K_\mtt m+K_\mtt n)    +  \nrm[T-U]_{X,\wh E}.
\end{equation*}
We use now that $\age(E)$ has the amalgamation property to find $V\in \age(E)_{\wh E}$ and $I\in \Emb(Z,V)$, $T_1\in \Emb((X,\mtt m),V)$ and $T_1\in \Emb((X,\mtt n), V)$ such that $\nrm[T_1- I\circ T_0]_{X_\mtt m, V},\nrm[U_1- I\circ U_0]_{X_\mtt n, V}\le \vep$.    
Thus,
 \begin{align*}
\nrm[T_1-U_1]_{X,V}= &\nrm[T_1-U_1]_{X,\wh E}\le  \nrm[T_1- I\circ T_0]_{X_\mtt m,  \wh E} K_\mtt m + \nrm[I\circ T_0-I\circ T_1]_{X,\wh E} + \\ 
+& 
 \nrm[U_1- I\circ U_0]_{X_\mtt m,  \wh E} K_\mtt n  
\le \vep (2+K_\mtt m +K_\mtt n)    +  \nrm[T-U]_{X,\wh E} .
\end{align*}
 Since $V\in \age(E)_{\wh E}$, and $\vep>0$ is arbitrary,  we obtain that $\partial_{X,E}(\mtt m,\mtt n)\le \nrm[T-U]_{X,\wh E}$.  Finally, because $\age(E)_{\wh E}$ is $\La_{\wh E}$-dense in $\age(\wh E)$, it follows that  $\bigcup_{Y\in \age(E)_{\wh E}}\mc L^k(X,Y)$ is dense in $\mc L^k(X,Y)$, so, $\mc N_X(E)$ is dense in $\mc N_X(\wh E)$.

{\it b)}:  $\tau_{k, \wh E}$ is continuous:  suppose that $V_n\to V$ for $n\to \infty$ in $\mr{Gr}(k,E)$  in the opening metric $\La_E$. Let $T\in \mc L^ k(\mbb F^k ,E)$ be such that $\im T=V$, and for each $i<k$ and  $n$ choose $x_i^{n}\in B_{(V_n,\nrm_E)}$ such that $\nrm[x_i^ n-T(u_i)]_E\to 0$ for $n\to \infty$. It is clear that, for $n$ large enough, $(T(x_i^n))_{i<k}$ are linearly independent, so the mapping $T_n: \mbb F^ k\to E$, $T_n(u_i)= x_i^n$, belongs to $\mc L^ k(\mbb F^k ,E)$ and   satisfies that $d_{X,E}(T_n-T)\to 0$ for $n\to \infty$, where $X=(\mbb F^k, \nu_{\mbb F^k, E}(T))$.  It follows from the continuity of $\nu_{X,E}$ that $\nu_{X,E}(T_n)\to \nu_{X,E}(T)$, so $\tau_{k,E}(V_n) \to_n \tau_{k,E}(V)$ for $n\to \infty$.

 Suppose that $\tau_{k,E}(V)=\tau_{k,E}(W)$. By the approximately ultrahomogeneity of $E$, for a given $\vep>0$ we can find  an isometry $g\in \Aut(E)$ such that $\La_E(V, g\cdot W)<\vep$, and hence $V\in [W]$.  The fact that  $\widehat\tau_{k,E}$ is a homeomorphism  follows from   {\it a)}.

{\it c)}: We start with the continuity of $\nu^2_{k,E}$. Suppose that $T_n\to_n T$ in norm. Then,     $\im (T_n)\to \im (T)$ in the opening distance $\La_E$. Now fix a basis  $(e_j)_{j<d}$ of $\im T$, and let $(x_j)_{j<k}$ be a linearly independent sequence in $E$ such that $T=T_0\circ T_1^*$, where $T_0:\mbb F^k\to E$ is  defined by $T_0(u_j)=e_j$ and $T_1: (\mbb F^k)^*\to E$  by $T_1(u_j^*)=x_j$.  For large enough $n$ choose a basis $\{e_j^n\}_{j<k}$ of $\im T_n$ such that $e_j^n\to_n e_j$   for every $j<k$.  Similarly, we define $T_0^n: \mbb F^k\to E$ , $T_0^n(u_j):= e_j^n$, and let $T_1^n:(\mbb F^k)^*\to E$ be such that $T_n=T_0^n\circ (T_1^n)^*$. Then  $T_0^n\to_n T_0$, so by continuity of $\nu_{\mbb F^k,E}$, it follows that $\nu_{\mbb F^k, E}(T_0^n)\to_n \nu_{\mbb F^k,E}(T_0)$.  On the other hand, $T_0$, $T_0^n$ are  1-1, so
$$\nrm[(T_1^n)^*-(T_1)^*]\le \nrm[T_0^{-1}]\cdot ( \nrm[T_0^n \circ (T_1^n)^*-T_0\circ T_1^* ]+   \nrm[T_0^n  -T_0]\nrm[ (T_0^n) ^{-1} ] \cdot \nrm[T] ). $$
This implies that $T_1^n\to_n T_1$, and $\nu_{(\mbb F^k)^*, E}(T_1^n)\to_n \nu_{(\mbb F^k)^*, E}(T_1)$. 

   Suppose that $\nu^2_{k,E}(T)=\nu^2_{k,E}(U)$.  Decompose $T=T_0\circ T_1^*$ and  $U=U_0 \circ U_1^*$  in a way that $\nu_{\mbb F^k, E}(T_0)=   \nu_{\mbb F^k, E}(U_0)$ and $\nu_{(\mbb F^k)^*, E}(T_1 )=   \nu_{(\mbb F^k)^*, E}(U_1)$.  As in the proof of {\it 1)}, we can find $g,h\in\iso(E)$  such that $\nrm[g\circ T_0 - U_0]\le \vep /(2 \nrm[T_1])$ and  $\nrm[h\circ T_1 - U_1]\le \vep/(2 \nrm[U_0])$. Hence,
 \begin{align*}
 \nrm[g \circ T\circ h^* - U]\le & \nrm[g \circ T_0 -U_0] \cdot\nrm[T_1] + \nrm[h\circ T_1  -U_1]\cdot \nrm[U_0] \le \vep.
 \end{align*}
 Since $\vep>0$ is arbitrary, $[U]=[T]$.  We see now that $\wt \nu^2_{k, E}$ is a homeomorphism.   Suppose that $\widetilde \nu^2_{k, E}([T_n])\to_n \widetilde \nu^2_{k, E}([T])\in \mc D_k(E)$. Our goal is to find a   subsequence of $([T_n])_n$ that converges to $[T]$: We  first decompose $T=T_0\circ T_1^*$, with $T_0 \in \mc L^k(\mbb F^k,E)$ and $T_1\in\mc L^k((\mbb F^k)^*,E)$,         a subsequence $(T_{n_m})_m$ and decompositions    $T_{n_m}=T_0^m\circ T_1^m$   in a way  that both $\om(\nu_{\mbb F^k,E} (T_0^m), \nu_{\mbb F^k,E} (T_0)) <m^{-1}$ and $ \om(\nu_{(\mbb F^k)^*,E} (T_1^m), \nu_{(\mbb F^k)^*,E} (T_1)) <m^{-1}$ for every $m\in \N$. It follows from {\it a)}  that  $[T_0^m]\to [T_0]$ and $[T_1^m]\to_m [T_1]$. This easily implies that $[T_{n_m}]\to_m [T]$.  The fact that $\mk d_{k,\wh E}$ is uniformly equivalent to $\wt{\om}_2$ on $\wt\om_2$-bounded sets follows from the Heine-Borel property of $(\mc D_k, \wt \om_2)$. 
\end{proof}

We finish with the following fact on bounded sets considered before.    
\begin{lemma}\label{jiowioerio43r446576}
  Suppose that $X=(X,\mtt m)$ is a normed space with $\dim X=k$, $E$ is a Banach space and  $\la\ge 1$. We have that 
   $\nu_{k,E}^2(\mc L_{\la}^{k,\mr{w}^*}(E^*,E))=\mc D_{k}(E;\la)$ and  $\nu_{k,E}^2(\mc L_{\sma\la}^{k,\mr{w}^*}(E^*,E))=\mc D_{k}(E; {\sma\la})$. 
 \end{lemma}  

\prue
 
We will use the following simple fact.
\clam \label{useful_dual}  
If  $V$ is a vector space with $\dim V=k$,   then 
$(\nu_{V, E}(T))^*(f)= \min\conj{\nrm[g]_{E^*}}{T^*(g)=f}$ for every $T\in \mc L^k(X,E)$ and $f\in X^*$.
\fclam
\prucl	
We know that $T:Y:=(V,\mtt n)\to E$  is an isometry  for $\mtt n:=\nu_{V,E}(T)$.  Fix $f\in X^*$, set  $Z:=(T(Y),\nrm_E)$,  and  $U: Z\to Y$ be the inverse of $T$. Let $g_0\in Z^*$ be such that $U^*(g_0)= f$, and let $g\in E^*$ be such that $\nrm[g]=\nrm[g_0]$. It is easily seen that $T^*(g)= f$, and since $\nrm[T^*]_{E^*,Y^*}=\nrm[T]_{Y,E}=1$, we obtain the desired equality. 
\fprucl
Fix $[(\mtt m_0,\mtt m_1)]\in \mc D_k(E;\la)$, and choose $T_0\in \mc L^k (\mbb F^k,E)$ and $T_1\in \mc L^k((\mbb F^k)^*,E)$ such that $\nu_{\mbb F^k,E}(T_0)=\mtt m_0$ and  $\nu_{(\mbb F^k)^*,E}(T_1)=\mtt m_1$. We claim that $T_0\circ T_1^*\in \mc L_\la^{k,\mr{w}^*}(E^*,E)$. Given $\nrm[g]_{E^*}=1$,
 $$\nrm[T_0 (T_1^*(g))]_E=\mtt m_0(T_1^*(g))\le \la \mtt m_1^*(T_1^*(g))\le  \la\nrm[g]_{E^*},$$ where the last inequality holds by Claim  \ref{useful_dual}.   Now suppose that $\nrm[T_0 (T_1^*(g))]_E\le \la^{-1}$. It follows that
 $\mtt m_0(T_1^*(g))\le \la^{-1}$, so $\mtt m_1^*(T_1^*(g))\le 1$. Hence, by Claim \ref{useful_dual}, there is $h\in E^*$ such that $T_1^*(h)=T_1^*(g)$ and $\nrm[h]_{E^*}\le 1$. This implies that $\mr{Ball}({\mr{Im}(T_0\circ T_1^*)})\con \la\cdot (T_0\circ T_1^*)(\mr{Ball}({E^*}))$. Similarly one shows that $\nu_{k,E}^2(\mc L_{\sma\la}^{k,\mr{w}^*}(E^*,E))=\mc D_{k}(E; {\sma\la})$.
 \fprue

\appendix
\section{Extrinsic metrics for $p=\infty$}  
The case $p=\infty$ is special because the  Fraïssé limit that corresponds to $\ell_\infty^\infty$  is a universal space, the Gurarij  space $\mbb G$. We are going to see that the $\mbb G$-extrinsic metrics are Lipschitz equivalent to the intrinsic ones on bounded sets. We start by analyzing $\partial_{X,\mbb G}$. Given a finite dimensional normed space $X=(X,\nrm_X)$, another compatible, more  geometrical, metric on  $\mc N_X$ is the next. Having in mind that a norm is completely determined by its dual unit ball, let  
    $$\alpha _{X}(\mtt m,\mtt n):=d_{\mathcal{H},\nrm_{X^{\ast}}}(\mr{Ball}({(X,\mtt m)^*}),\mr{Ball}({(X,\mtt n)^*})),$$
     where   $d_{\mathcal{H},\nrm_{X^{\ast }}}(\cdot,\cdot )$ is the Hausdorff distance with respect to the norm distance induced by $\nrm_{X^*}$. In other words, $\al_X(\mtt m, \mtt n)$ measures the $d_{\nrm_X^*}$-distance between the   unit balls of $(X^*,\mtt m^*)$ and of  $(X^*,\mtt n^*)$.  In the next  we write $\mr{Sph}(X)=\conj{x\in X}{\nrm[x]_X=1}$ to denote the unit sphere of $X$.

\prop \label{partial_is_alpha_gurarij}
Let  $X=(X,\nrm_X)$ be a finite dimensional normed space and let $\mtt m,\mtt n\in \mc N_X$.
\begin{enumerate}[a)]
\item If  $\mtt m,\mtt n\in \mc N_X(\ell_\infty^\infty)$, then   $\partial_{X,\ell_\infty^\infty}(\mtt m,\mtt n)= \alpha _{X}(\mtt m,\mtt n)$.  Consequently, in general, $\partial_{X,\mbb G}(\mtt m,\mtt n)= \alpha _{X}(\mtt m,\mtt n)$.
\item   If $\mtt m,\mtt n \in B_\om[\nrm_X;\la]$,  then $\la^{-1}\cdot\omega(\mtt m,\mtt n)\le \alpha_X(\mtt m,\mtt n)  \le \la \cdot
\omega(\mtt m,\mtt n)$.
   
\end{enumerate}   
 
\fprop

\begin{proof}
{\it a)}: Fix $\vep>0$, and let $T,U\in \mc L(X,\ell_\infty^\infty)$ be such that $\nu_{X,\ell_\infty^\infty}(T)=\mtt m$ and $\nu_{X,\ell_\infty^\infty}(U)=\mtt n$, and  $\nrm[T-U]_{X,\ell_\infty^\infty}\le \partial_{X,\ell_\infty^\infty}(\mtt m,\mtt n)+\vep$.  Given   $f\in \mr{Ball}({(X,\mtt m)^*})$, if $g\in \mr{Ball}({(\ell_\infty^\infty)^*})$ is such that $T^*(g)=f$, then
$d_{X^*}(f, \mr{Ball}({(X,\mtt n)^*})) \le \nrm[f - U^*(g)]_{X^*}\le \nrm[T^*-U^*]\le \partial_{X,\ell_\infty^\infty}(\mtt m,\mtt n)+\vep$.  We  show that $\partial_{X,\ell_\infty^\infty}(\mtt m,\mtt n)\le \al_X(\mtt m,\mtt n)$:    Choose $n\in \N$ such that $\im T, \im U\con \langle u_j \rangle_{j<n}$. We identify canonically $\langle u_j \rangle_{j<n}$ with the induced $\sup$-norm with  $\ell_\infty^n$. For every $0\le  j< n$ choose $f_{j},g_{j}\in \mr{Ball}(\ell_1^n)$ such
that $\Vert T^{\ast }(u_{j}^{\ast})-U^{\ast }(g_{j})\Vert _{X^{\ast}}=  d_{X^{\ast }}(T^{\ast }(u_{j}^{\ast }),\ball((X,\mtt m)^{\ast }))$ and $\Vert U^{\ast }(u_{j}^{\ast })-T^{\ast }(f_{j})\Vert _{X^{\ast}}= d_{X^{\ast }}(U^{\ast }(u_{j}^{\ast }),\ball((X,\mtt n)^{\ast }))$.
Let $\xi ,\eta \in \mathrm{Emb}(\ell_\infty^n,\ell_\infty^{2n})$ be defined dually by $\xi ^{\ast }(u_{j}^{\ast
}):=u_{j}^{\ast }$, $\xi ^{\ast }(u_{n+j}^{\ast }):=f_{j}^{\ast }$ , $\eta ^{\ast }(u_{j}^{\ast }):=g_{j}$ and $\eta ^{\ast }(u_{n+j}^{\ast
})=u_{j}^{\ast }$ for $0\leq j< n$. Then,
\[\partial_{X,\ell_\infty^\infty}(\mtt m,\mtt n)\le \nrm[\xi \circ T- \eta\circ U]_{X,\ell_\infty^{2n}} =\nrm[T^* \circ \xi^*- U^*\circ \eta^*]_{\ell_1^{2n},X^*}=\al_X(\mtt m,\mtt n).      \]

{\it b)}:  
As  $\omega(\mtt m,\mtt n)=\omega(\mtt m^*,\mtt n^*)$, it suffices to prove   $\la^{-1}\cdot\omega(\mtt m,\mtt n)\le d_{\mathcal{H},\nrm_{X}}(\ball(X,\mtt m),\ball(X,\mtt n))  \le \la \cdot \omega(\mtt m,\mtt n)$ provided that  $\om(\nrm_X, \mtt m),\om(\nrm_X,\mtt n)\le \la$. We  assume that $d_{\mathcal{H},\nrm_X}(\mtt m,\mtt n)>0$.  Let us show the first inequality. Without of generality we assume that there is $x\in \mr{Sph}(X,\mtt m)$ such that $0<d_{\mathcal{H},\nrm_X}(\ball(X,\mtt m),\ball(X,\mtt n))=d_{X}(x, \ball(X,\mtt n))$.  Then $\mtt n(x)>1$ and
$d_{\mathcal{H},\nrm_X}(\mtt m,\mtt n)\le  \nrm[x- {x}/\nrm[x]_X]_X\le \la | 1-
{\mtt n(x)^{-1}}| =\la(1-{\mtt n(x)^{-1}}) \le
\la(1-\exp({-\omega(\mtt m,\mtt n)}))\le \la\omega(\mtt m,\mtt n)$. Suppose now that $\mtt m(x)=1$. Let $y\in X$ be such that $\mtt n(y)\le 1$ and $\nrm[x-y]_X\le d_{\mathcal{H},\nrm_X}(\mtt m,\mtt n)$. It follows that
$\mtt n(x)\le \mtt n(y)+\mtt n(x-y)\le 1+ \la \nrm[x-y]_X\le 1+\la d_{\mathcal{H},\nrm_X}(\mtt m,\mtt n) \le \exp({\la\cdot d_{\mathcal{H},\nrm_X}(\mtt m,\mtt n)})$. Consequently, $\omega(\mtt m,\mtt n)\le \la d_{\mathcal{H},\nrm_X}(\mtt m,\mtt n)$.
\end{proof}

 We see now that $\ga_\mbb G$ is Lipschitz equivalent to the Banach-Mazur metric on $\mc B_k$.

\cor
 $d_\mr{BM}$ and $\ga_{k,\mbb G}$ are Lipschitz equivalent on $\mc B_k$. In fact, for  $\mtt m,\mtt n\in \mc N_k$,
 $$\frac1{4 k \log k } d_\mr{BM}([\mtt m],[\mtt n])\le \ga([\mtt m],[\mtt n])\le (\log k) d_\mr{BM}([\mtt m],[\mtt n]).$$
\fcor
\prue  We start with the following.
\clam
 $d_\mr{BM}([\mtt m],[\mtt n]) \le 4 k \log k \ga_{k,E}([\mtt n],[\mtt n])$  for every Banach space $E$ and  $\mtt m,\mtt n\in \mc N_k(E)$.
\fclam
\prucl
Suppose  that $\ga_{k,E}([\mtt m],[\mtt n])<1/(3  k)$. Let $V,W\in \mr{Gr}(k,E)$ be such that $\tau_{k,E}(V)=[\mtt m]$, $\tau_{k,E}(W)=[\mtt n]$, and $\ga_E([\mtt m],[\mtt n])=\La_E(V,W)$.  Let
$(x_j)_{j<k}$ be an Auerbach basis of $(V,\nrm_E)$. For each $j<k$, let $y_j\in \mr{Ball}({(V,\nrm_E )})$ be
such that $\nrm[x_j-y_j]_E\le \La_E(V,W)$. Since
\begin{align}
\nrm[\sum_{j<k} \la_j y_j]_E   \ge   &  \nrm[\sum_{j<k} \la_j x_j]_E - \nrm[\sum_{j<k} \la_j (x_j-y_j)]_E\ge    \nrm[\sum_{j<k} \la_j x_j]_E - k \La_E(V,W)\max_{j<k}|\la_j| \ge  \nonumber  \\
\ge  & (1- k \La_E(V,W)) \nrm[\sum_{j<k} \la_j x_j]_E>0. \label{oi8978444}
\end{align}
we obtain that  $(y_j)_{j<k}$ is a basis of $W$ and $T:V\to W$, 
$T(x_j):=y_j$, $j<k$ is invertible.  In addition, from  \eqref{oi8978444} we have that $\nrm[T^{-1}]_{(W,\nrm_E), (V,\nrm_E)}\le (1- k
\La_E(V,W))^{-1}  $, and similarly, $\nrm[T]_{(V,\nrm_E),(W,\nrm_E)}\le 1+k
\La_E(V,W)$.  We use that $(1+x) / (1-x) \le \exp(9 x/4)$ for every $0\le x\le 1/3$,   to conclude   that $ d_\mr{BM}(\tau_{k,E}(V), \tau_{k,E}(W))\le ({9}/{4}) k \La_E(V,W)\le 4k \log (k) \ga_{k,E}([\mtt m],[\mtt n]).$ 
Suppose that $\La_E(V,W)\ge(3k)^{-1}$. Since the diameter of $\mathcal B_k$ is at most $\log(k)$, we
obtain that $d_\mr{BM}(\tau_{k,E}(V), \tau_{k,E}(W)) \le \log(k) \le  3k \log(k)\La_E(V,W)$.
In any case, $d_\mr{BM}(\tau_{k,E}(V), \tau_{k,E}(W)) \le  4k \log(k)\La_E(V,W)$.
\fprucl
Fix two norms $\mtt m,\mtt n\in \mc N_k$, set $X:=(\mbb F^k, \mtt m)$ and $Y:=(\mbb F^k, \mtt n)$. The following  result is a slight modification of \cite[Proposition 6.2]{Ostrovski_topologies_1994}.
\clam\label{ostrow}
 Suppose that $E$ and $F$ are two  finite-dimensional  normed spaces, and  $T:F\to G$ is a 1-1 linear operator. There is a   normed space $H$  and $I\in \Emb(F,H)$ and $J\in \Emb(G,H)$ such that:
  \begin{enumerate}[i)]
\item  If $1\le \nrm[T], \nrm[T^{ -1}]$, then    $\nrm[I -J\circ T]\le \nrm[T] \cdot \nrm[T^{-1} ] -1$.
\item If $\dim F=\dim G$  and $\nrm[T]=1$, then  $\La_H( \im I, \im J)\le  \nrm[T^{-1}]-1$.
 \end{enumerate}
\fclam
 \prucl
Fix a 1-1 linear operator $T:F\to G$. On the direct sum $F\oplus G$ we define the seminorm
$$\mtt m(x,y):=\max\left\{\nrm[\frac{Tx}{\nrm[T]}+y]_G, \max_{g\in D} \left| \frac{g}{\nrm[T]}(y) + \frac{ (T^* g)(x)  }{\nrm[T^* g]_{F^*}   }  \right|       \right\},$$
where $D$ is chosen so that $ (T^*)^{-1}(\nrm[T^{-1}]^{-1}\cdot\mr{Ext}(\mr{Ball}({F^*})))\con    D\con  \mr{Ball}(F)$, and where for a compact convex set $K$, $\mr{Ext}(K)$ is the set of extreme points of $K$.  
Let $H$ by the quotient of $F\oplus G$ by the kernel of $\mtt m$, and  let  $I:F\to H$, $J: G\to H$ be the two canonical injections   $I(x):=[(x,0)]$,  $J(y):=[(0,y)]$.   It is routine to check that $I,J$ and $H$ have the desired properties.
\fprucl
Let $T:X\to Y$ be such that $\nrm[T]\cdot \nrm[T^{-1}]=\exp(d_\mr{BM}([\mtt m],[\mtt n]))$, and without loss of generality, we assume that $\nrm[T]=1$.
 We apply Claim \ref{ostrow} to $T$, and  we obtain a normed space $Z$ and isometric embeddings $I:X\to Z$ and $J:Y\to Z$ such that (b) holds, that is, $\La_Z(\im I, \im J)\le \exp(d_\mr{BM}([\mtt m],[\mtt n]))-1$.  Since $d_{BM}([\mtt m],[\mtt n])\le \log k$, it follows that $\exp(d_\mr{BM}([\mtt m],[\mtt n]))-1\le \log k\cdot d_\mr{BM}([\mtt m],[\mtt n])$. Thus $\ga([\mtt m],[\mtt n]) \le \La_Z(X,Y) \le \log k \cdot d_\mr{BM}([\mtt m],[\mtt n])$.
\fprue

We conclude by proving that    $\mk d_{k,\mbb G}$ is Lipschitz equivalent to the following intrinsically  defined metric on $\mc D_k(\la)$.  Recall that  $\om_2$ is  the compatible metric $\om_2((\mtt m_0,\mtt m_1),(\mtt n_0,\mtt n_1)):=\om(\mtt m_0,\mtt n_0)+\om(\mtt m_1,\mtt n_1)$, and  that  $\wt\om_2$ is the corresponding quotient metric on $\mc D_k$.

\prop \label{ga_in_gurarij}
$\mk d_{k,\ell_\infty^\infty}$ and $\wt w_2$ are Lipschitz equivalent on $\mc D_k(\ell_\infty^\infty;\la)$.  In fact, 
$$\frac1{2 \log(\la k) \la \sqrt k }\wt\om_2([\mbf m],[\mbf n])  \le   \mk d_{k,\ell_\infty^\infty}([\mbf m],[\mbf n])\le k^2 \la^3 \wt\om_2([\mbf m],[\mbf n]).$$
\fprop
\prue
We first estimate the $\wt{\om}_2$-diameter of $\mc D_k(\ell_\infty^\infty,\la)$.
\clam\label{diam_of_D}
 For every $[(\mtt m_0,\mtt m_1)],[(\mtt n_0,\mtt n_1)]\in \mc D_k(\ell_\infty^\infty;{\la})$ one has that
 \begin{equation}\label{oiu43u7843y54378543543}
\wt{\om}_2([(\mtt m_0,\mtt m_1)],[(\mtt n_0,\mtt n_1)])\le 2(\log \la+ \min\{d_\mr{BM}([\mtt m_0], [\mtt n_0]), d_\mr{BM}([\mtt m_1],[\mtt n_1])\}).   
 \end{equation}  
Consequently, $\mr{diam}(\mc D_k(\ell_\infty^\infty;\la))\le 2\log( \la k )$. 
\fclam
 \prucl
 Given $[(\mtt m_0,\mtt m_1)],[(\mtt n_0,\mtt n_1)]\in \mc D_k({\ell_\infty^\infty;\la})$, $\om(\mtt m_0, \mtt m_1^*), \om(\mtt n_0,\mtt n_1^*)\le \log(\la)$.
 Choose $\De\in \mr{GL}(\mbb F^k)$  with $d_\mr{BM}([\mtt m_0],[\mtt n_0])=\om( \mtt m_0,\De\cdot  \mtt n_0)$. Then, $\om(\mtt m_1,\De\cdot \mtt n_1)\le \om(\mtt m_1, \mtt m_0^*) +\om(\mtt m_0^*,\De\cdot \mtt n_0^*) +\om(\De \cdot \mtt n_0^*,\De\cdot \mtt n_1)\le 2\log \la +d_\mr{BM}(\mtt m_0, \mtt n_0)$.  From here we get  easily the inequality in \eqref{oiu43u7843y54378543543}.
 \fprucl

\clam\label{useful_omega_2}
Let $E$ be  a Banach space. If $T,U\in \mc L_{\la}^{k,\mr{w}^*}(E^*,E)$ are such that $ \nrm[T-U]< 1/(\la\sqrt k)$, then, $\wt{\om}_2(\nu_{k,E}(T),\nu_{k,E}(U))\le  \la\sqrt k \nrm[T-U]$.
  \fclam
\prucl
 Let $T_0, U_0\in \mc L^k(\mbb F^k, E)$ and $T_1,U_1\in \mc L^k( (\mbb F^k)^*,E)$ be such that $T=T_0\circ T_1^*$ and $U=U_0\circ U_1^*$. Define $\mtt m_0:=\nu_{\mbb F^k,E}(T_0)$, $\mtt m_1:=\nu_{(\mbb F^k)^*,E}(T_1)$, and $\mtt n_0:=\nu_{\mbb F^k,E}(U_0)$, $\mtt n_1:=\nu_{(\mbb F^k)^*,E}(U_1)$.   We use the \emph{Kadets-Snobar Theorem}---see for example \cite[Theorem 12.1.6]{albiac_topics_2006}---to fix a projection $P:E\to \im(T_1)$ of norm at most $\sqrt{k}$.   Then, $T_1^*= T_1^*\circ P^*\circ r$, where $r:E^*\to (\im T_1)^*$ is the restriction map $r(g):=g\upharpoonright  {\im T_1}$, hence, the rank of $T_1^{*}\rest  {\im P^*}$ is $k$ and since the dimension of $\im P^*$ is also $k$, it follows that $\theta:=T_1^*:   \im P^* \to (\mbb F^k, \mtt m_0)$ is an isomorphism. We estimate some norms.

\begin{enumerate}[$\bullet$]
\item   $\nrm[\theta]\le \la$ and $\nrm[\theta^{-1}]\le  \la\sqrt{k} $:  $\nrm[\theta]\le \nrm[T_1^*]_{E^*, (\mbb F^k, \mtt m_0)}=\nrm[T ]_{E^*,E}\le \la$.  
Now fix $g\in \im P^*$, and suppose that  $1=\mtt m_0(T_1^*(g))=\nrm[T(g)]_E$.
Find $h\in E^*$ with $\nrm[h]_{E^*}\le \la$ such that $T^*(h)=T^*(g)$; then $\theta(g)=T_1^*(g)=T_1^*(h)=T_1^*(h_0)=\theta(h_0)$, where $h_0=P^*(h\upharpoonright  {\im T_1})$. Since $\theta$ is a bijection, $g=h_0$ and $\nrm[g]_{E^*}\le \nrm[P^*]\nrm[h]_{E^*} \le \la\sqrt{k} $. 

\item    $U_1^*: \im P^* \to (\mbb F^k, \mtt n_0)$ is also isomorphism: First, $\nrm[U_1^*]_{E^*,(\mbb F^k,\mtt n_0)}=\nrm[U]_{E^*,E}\le \la$, and if $g\in \mr{Sph}(\im P^*)$, then by (a), $\mtt m_0(T_1^*(g))\ge 1/(\la \sqrt{k})$, and 
$|\mtt n_0(U_1^*(g))-\mtt m_0(T_1^*(g))|=  |\nrm[Tg]_E-\nrm[U(g)]_E|\le \nrm[T-U]<{1}/({\la\sqrt k}) $.
So,  $\mtt n_0(U_1^*(g))>   0$, hence $U_1^*(g)\neq 0$, and since $\dim \im P^*=k$,  it follows that $U_1^*$ is an isomorphism. Set  $\De:= U_1^*\circ (T_1^*)^{-1} \in \mr{GL}(\mbb F^k)$.

\item  $\De \cdot  \mtt m_1=\mtt n_1$:  If $f\in (\mbb F^k)^*$, then   $(\De\cdot \mtt m_1)(f)= \mtt m_1 (\De^*(f))=\nrm[T_1(T_1^{-1} (U_1(f)))]_E=\nrm[U_1(f)]_E=\mtt n_1(f)$.
\item $ \om(\De\cdot \mtt m_0,  \mtt n_0)\le  \la \sqrt k \nrm[T-U]$:     Fix $x\in \mbb F^k$ such that $\mtt m_0(x)=1$, and set $g:=(U_1^*)^{-1}(x)$.  
Notice that $\nrm[g]_{E^*}\le \la\sqrt k$. Then
$|\mtt m_0(\De^{-1}(x))- \mtt n_0(x)|= | \nrm[T_0(T_1^*(g))]_E - \nrm[U_0(U_1^*(g)]_E|\le \nrm[T-U]_{E^*,E}\nrm[g]_{E^*}\le \la\cdot \sqrt k \cdot \nrm[T-U]_{E^*,E}.$
So, $(\De\cdot\mtt m_0)(x)\le (\la\cdot \sqrt k \cdot \nrm[T-U] +1) \mtt n_0(x)$. Similarly,  $\mtt n_0(x)\le (\la\cdot \sqrt   k \cdot \nrm[T-U] +1) (\De\cdot \mtt m_0)(x)$. It follows that
\begin{align*}\pushQED{\qed}
\wt{\om}_2([(\mtt m_0,\mtt m_1)],[(\mtt n_0,\mtt n_1)])\le & {\om}_2(\De\cdot(\mtt m_0,\mtt m_1),(\mtt n_0,\mtt n_1))= \om(\De\cdot \mtt m_0, \mtt n_0) \le \\
\le & \log (\la\cdot \sqrt k \cdot \nrm[T-U] +1) \le \la\cdot \sqrt k \nrm[T-U].    \qedhere
\end{align*}
\end{enumerate}

\fprucl 
From the previous two  claims  we have that 
\(\wt\om_2([\mbf m],[\mbf n])  \le 2 \log(\la k) \la \sqrt k
\mk d_{k,E}([\mbf m],[\mbf n])\)  for  $[\mbf m],[\mbf n]\in \mc D_k(E;\la)$.
 \clam
For every $[\mbf m],[\mbf n]\in \mc D_k(\ell_\infty^\infty;\la)$ one has that $\mk d_{k,\ell_\infty^\infty}([\mbf m],[\mbf n])\le k^2 \la^3
 \wt\om_2([\mbf m],[\mbf n])$.
\fclam
\prucl
Let $\De\in \mr{GL}(\mbb F^k)$ be such that $\wt\om_2([(\mtt m_0,\mtt m_1)],[(\mtt n_0,\mtt n_1)])= \om(\mtt m_0, \De\cdot \mtt n_0)+ \om(\mtt m_1, \De\cdot \mtt n_1)$. We use Proposition \ref{partial_is_alpha_gurarij}  to find $T_0,U_0\in \mc L(\mbb F^k,\ell_\infty^\infty)$, $T_1,U_1\in \mc L((\mbb F^k)^*,\ell_\infty^\infty)$ such that:
\begin{enumerate}[$\bullet$]
\item $\mtt m_0=\nu_{\mbb F^k,\ell_\infty^\infty}(T_0)$, $\De\cdot \mtt n_0=\nu_{\mbb F^k,\ell_\infty^\infty}(U_0)$, $\mtt m_1=\nu_{(\mbb F^k)^*,\ell_\infty^\infty}(T_1)$ and $\De\cdot \mtt m_1=\nu_{(\mbb F^k)^*,\ell_\infty^\infty}(U_1)$;
\item  $\al_{(\mbb F^k, \mtt m_1^*)}(\mtt m_0,\De\cdot \mtt n_0)=\nrm[T_0-U_0]_{(\mbb F^k, \mtt m_1^*), \ell_\infty^\infty}$ and
$\al_{((\mbb F^k)^*, \mtt m_0^*)}(\mtt m_1,\De\cdot \mtt n_1)=\nrm[T_1-U_1]_{((\mbb F^k)^*, \mtt m_0^*), \ell_\infty^\infty}$.
\end{enumerate}
Let $T:= T_0\circ T_1^*$, $U:=U_0\circ U_1^*$. It follows that $\nu^2([T])=[\mbf m]$, $\nu^2([U])=[\mbf n]$, and $\mk d_{k,\ell_\infty^\infty}([\mbf m], [\mbf n])\le \nrm[T-U]_{\ell_1^\infty,\ell_\infty^\infty}$. Now let $g\in \mr{Sph}(\ell_1^\infty)$. Then,
\begin{align}
\nrm[(T-U)(g)]_{\ell_\infty^\infty}\le & \nrm[T_0 (T_1^*-U_1^*)(g)]_{\ell_\infty^\infty} +\nrm[(T_0-U_0)(U_1^*(g))]_{\ell_\infty^\infty} \le  
\  \nrm[T_0]_{(\mbb F^k, \mtt m_0),\ell_\infty^\infty} \nrm[T_1^*-U_1^*]_{\ell_1^\infty, (\mbb F^k, \mtt m_0)} + \nonumber \\
+&  \nrm[T_0-U_0]_{(\mbb F^k, \mtt m_1)^*,\ell_\infty^\infty}\nrm[U_1^*]_{\ell_1^\infty, (\mbb F^k, \mtt m_1)^*}=   
 \al_{(\mbb F^k, \mtt m_0)^*}(\mtt m_1,\De\cdot \mtt n_1)+\al_{(\mbb F^k, \mtt m_1)^*}(\mtt m_0,\De\cdot \mtt n_0) \label{kjnjkuygyuyue}
\end{align}
 On the other hand, by  Claim \ref{diam_of_D},
 \begin{align*}
 \max \{ \nrm[\id]_{\De\cdot \mtt n_1,\mtt m_0^*}, \nrm[\id]_{\De\cdot \mtt n_0,\mtt m_1^*} \}\le &
 \max\{\nrm[\id]_{\mtt m_1,\mtt m_0^*}\cdot  \nrm[\id]_{\De\cdot \mtt n_1,\mtt m_1},  \nrm[\id]_{\mtt m_0,\mtt m_1^*}\cdot \nrm[\id]_{\De\cdot \mtt n_0,\mtt m_0}  \} \le \\
 \le & \la e^{\wt\om_2([\mbf m], [\mbf n])}\le k^2 \la^3.
 \end{align*}
It follows from  the inequality in \eqref{kjnjkuygyuyue} and  Proposition  \ref{partial_is_alpha_gurarij} {\it b)} that
\[\pushQED{\qed}
  \mk d_{k,\ell_\infty^\infty}([\mbf m],[\mbf n])\le \nrm[T-U]_{\ell_1^\infty,\ell_\infty^\infty}\le   k^2 \la^3 ( \om( \mtt m_1,\De\cdot \mtt n_1)+\om(\mtt m_0,\De\cdot \mtt n_0 )) \le k^2 \la^3 \wt\om_2([\mbf m],[\mbf n]). \qedhere
\popQED  \popQED 
\]\let\qed\relax
 \fprucl 
\let\qed\relax

\fprue

We do not know  a similar explicit description  of the extrinsic metrics for $p$'s other than $ \infty$.

  \nocite{spencer_1979}
\nocite{nesetril_handbook}

\bibliographystyle{amsplain}
\bibliography{bibliography}

\end{document}